\documentclass[reqno, 10pt]{amsart}
\usepackage{upgreek}
\usepackage{mathtools}
\usepackage{amssymb}
\usepackage{mathrsfs}
\usepackage{amsthm}
\numberwithin{equation}{section} 

\usepackage{times}

\usepackage{soul}

\usepackage{textcomp}

\usepackage[scr=boondox]{mathalpha}
\usepackage{xcolor}
\usepackage{srcltx}
\usepackage{thmtools}
\usepackage[left=1in, right=1in, top=1.25in, bottom=1.25in]{geometry}

\usepackage{tikz}
\usetikzlibrary{shapes,arrows,arrows.meta}

\usepackage{multicol}
\usepackage{graphicx}
\usepackage[colorlinks,allcolors=blue!75!black,urlcolor=black]{hyperref}
\usepackage{cancel}

\makeatletter
\newcommand{\leqnomode}{\tagsleft@true\let\veqno\@@leqno}
\makeatother

\allowdisplaybreaks

\newcommand{\bbC}{\mathbb{C}}

\newcommand{\bbL}{\mathbb{L}}

\newcommand{\bbV}{\mathbb{V}}

\def\bA{\mathbf{A}}
\def\bB{\mathbf{B}}
\def\bC{\mathbf{C}}

\def\bH{\mathbf{H}}
\def\bI{\mathbf{I}}
\def\bL{\mathbf{L}}
\def\bM{\mathbf{M}}
\def\bN{\mathbf{N}}
\def\bP{\mathbf{P}}

\def\bV{\mathbf{V}}
\def\bW{\mathbf{W}}

\def\bb{\mathbf{b}}

\def\be{\mathbf{e}}
\def\bg{\mathbf{g}}
\def\bh{\mathbf{h}}
\def\bj{\mathbf{j}}

\def\bu{\mathbf{u}}
\def\bv{\mathbf{v}}

\def\bx{\mathbf{x}}
\def\by{\mathbf{y}}

\renewcommand{\bf}{\mathbf{f}}

\def\bchi{\boldsymbol{\chi}}

\def\bvarphi{\boldsymbol{\varphi}}

\def\bnu{\boldsymbol{\nu}}
\def\bxi{\boldsymbol{\xi}}

\def\bsfA{\boldsymbol{\mathsf{A}}}
\def\bsfb{\boldsymbol{\mathsf{b}}}
\def\bsfc{\boldsymbol{\mathsf{c}}}
\def\sfd{\mathsf{d}}

\newcommand{\wh}{\widehat}
\newcommand{\wt}{\widetilde}

\newcommand{\LL}{\mathrm{L}}
\newcommand{\HH}{\mathrm{H}}
\newcommand{\NN}{\mathrm{N}}
\newcommand{\CC}{\mathrm{C}}
\newcommand{\VV}{\mathrm{V}}
\newcommand{\RR}{\mathrm{R}}
\newcommand{\WW}{\mathrm{W}}

\newcommand{\cJ}{\mathcal{J}}

\newcommand{\cU}{\mathcal{U}}
\newcommand{\cV}{\mathcal{V}}
\newcommand{\cW}{\mathcal{W}}
\newcommand{\cX}{\mathcal{X}}

\newcommand{\sL}{\mathscr{L}}
\newcommand{\sG}{\mathscr{G}}
\newcommand{\sU}{\mathscr{U}}
\newcommand{\sV}{\mathscr{V}}

\newcommand{\dd}[1][y]{\if#1y\,\fi{\mathrm d}}

\renewcommand{\div}{\operatorname{div}}

\newcommand{\tr}{\operatorname*{tr}}

\newcommand{\G}{\Gamma}
\newcommand{\nablaG}{\nabla_\G}
\newcommand{\nablaM}{\nabla_M}
\newcommand{\divG}{\div_\G}
\newcommand{\divM}{\div_M}
\newcommand{\DeltaG}{\Delta_\G}

\newcommand{\uD}{\underline{D}}

\theoremstyle{plain}

\newtheorem{thm}{Theorem}[section]
\newtheorem{proposition}[thm]{Proposition}
\newtheorem{lemma}[thm]{Lemma}
\newtheorem{corollary}[thm]{Corollary}
\newtheorem{definition}[thm]{Definition}
\newtheorem{remark}[thm]{Remark}

\theoremstyle{definition}

\theoremstyle{remark}

\def\itemautorefname~#1\null{%
  (#1)\null
}

\newenvironment{proof-thm}[1]{\vskip1em \noindent\textbf{Proof of \autoref#1.}}{\hfill$\square$\\}

\makeatletter
\renewcommand\subsection{\@startsection{subsection}{2}%
  \z@{.5\linespacing\@plus.7\linespacing}{.1\linespacing}%
  {\normalfont\centering\scshape}}
\makeatother

\makeatletter
\renewcommand\subsubsection{\@startsection{subsubsection}{3}%
  \z@{.5\linespacing\@plus.7\linespacing}{.1\linespacing}%
  {\normalfont\raggedright\scshape}}
\makeatother

\makeatletter
\def\@tocline#1#2#3#4#5#6#7{\relax
  \ifnum #1>\c@tocdepth 
  \else
    \par \addpenalty\@secpenalty\addvspace{#2}%
    \begingroup \hyphenpenalty\@M
    \@ifempty{#4}{%
      \@tempdima\csname r@tocindent\number#1\endcsname\relax
    }{%
      \@tempdima#4\relax
    }%
    \parindent\z@ \leftskip#3\relax \advance\leftskip\@tempdima\relax
    \rightskip\@pnumwidth plus4em \parfillskip-\@pnumwidth
    #5\leavevmode\hskip-\@tempdima
      \ifcase #1
       \or\or \hskip 1em \or \hskip 2em \else \hskip 3em \fi%
      #6\nobreak\relax
    \dotfill\hbox to\@pnumwidth{\@tocpagenum{#7}}\par
    \nobreak
    \endgroup
  \fi}
\makeatother

\begin{document}
\title[$\LL^p$-based Scalar Elliptic Theory on closed minimally regular manifolds]{$\LL^p$-based Sobolev theory on closed manifolds of minimal regularity:\\ Scalar Elliptic Equations}

\author[G.A.~Benavides]{Gonzalo A.~Benavides}

\author[R.H.~Nochetto]{Ricardo H.~Nochetto}

\author[M.~Shakipov]{Mansur Shakipov}

\address{Department of Mathematics, University of Maryland College Park, MD 20742}
\email{gonzalob@umd.edu, rhn@umd.edu, shakipov@umd.edu}


\begin{abstract}
This paper and its follow-up \cite{BenavidesNochettoShakipov2025-b} are concerned with the well-posedness and $\LL^p$-based Sobolev regularity for appropriate weak formulations of a family of prototypical PDEs posed on manifolds of minimal regularity.
In particular, the domains are assumed to be compact, connected $d$-dimensional manifolds without boundary of class $C^k$ and $C^{k-1,1}$ ($k \geq 1$) embedded in $\RR^{d+1}$.
The focus of this program is on the $\LL^p$-based theory that is sharp with respect to the regularity of the source terms and the manifold.
In the present paper, we focus our attention on the case of general scalar elliptic problems.
We first establish $\LL^p$-based well-posedness and higher regularity for the purely diffusive problems with variable coefficients by localizing and rewriting these equations in flat domains to employ the Calder\'{o}n--Zygmund theory, combined with duality arguments. We then invoke the Fredholm alternative to derive analogous results for general scalar elliptic problems, underscoring the subtle differences that the geometric setting entails compared to the theory in flat domains.
\end{abstract}

\maketitle
\tableofcontents

\section{Introduction}\label{sec:intro}

Let $\G$ be a $d$-dimensional ($d \geq 2$) compact Lipschitz manifold without boundary embedded in $\RR^{d+1}$ and denote by $\bnu = (\nu_i)_{i=1}^{d+1}$ its outward unit normal.
Throughout this work, we are concerned with developing $\LL^p$-based well-posedness and regularity theory for the \textit{general scalar elliptic problem} posed on $\G$:
\begin{equation}\label{eq:gsep}
- \divG (\bsfA \nablaG u + u\, \bsfb ) + \bsfc \cdot \nablaG u + \sfd u = f,
\end{equation}
where $f$ is a given scalar field, and $\bsfA = (A_{ij})_{i,j=1}^{(d+1) \times (d+1)}$, $\bsfb = (b_i)_{i=1}^{d+1}$, $\bsfc = (c_i)_{i=1}^{d+1}$ and $\sfd$ are given coefficients.
The operators $\nablaG$ and $\divG$ respectively denote the \emph{covariant gradient} and \emph{covariant divergence}; see \autoref{sec:diff-geo}.
The unknown $u$ in \eqref{eq:gsep} is a scalar function $u:\G \rightarrow \RR$.
For $(\bsfA, \bsfb, \bsfc, \sfd) = (\bI, \mathbf{0}, \mathbf{0},0)$, problem \eqref{eq:gsep} reduces to the \textit{Laplace--Beltrami} problem
\begin{equation} \label{eq:lb}
-\DeltaG u = f,
\end{equation}
where $\DeltaG := \divG \nablaG$ is called the Laplace--Beltrami operator.

Problem \eqref{eq:gsep} is of fundamental importance as a prototypical constituent of the modeling of a great variety of physical phenomena on thin films, and in the analysis of their underlying PDEs.
Relevant applications of the general scalar elliptic problem include surfactant modeling \cite{surfactant, surfactant-2}, modeling of two-phase geometric biomembranes \cite{ElliottStinner2010-2}, diffusion-induced grain boundary motion \cite{grainBoundaryMotion, grainBoundaryMotion-2}, and color image denoising \cite{BeltramiDenoising}, to name a few.
Beyond modeling and applications, \eqref{eq:gsep} is essential in the analysis of evolving-surface parabolic problems \cite{KovacsLi2022}, wherein the PDE posed on a moving surface $\G(t)$ is pulled back to the initial surface $\G(0)$, thereby yielding an elliptic term with a diffusion matrix $\bsfA(t)$ that depends on $\G(t)$.

The purpose of this work is threefold:
\begin{itemize}
\item to provide a careful review of various definitions of Sobolev spaces defined on $\G$ that is less regular than $C^2$, and establishing the conditions for them to coincide.
\item to establish the $\LL^2$- and $\LL^p$-based well-posedness of \eqref{eq:gsep} for $\G$ of class $C^{0,1}$ and $C^1$, respectively, under mild assumptions on the coefficients $(\bsfA, \bsfb, \bsfc,\sfd)$.
\item to establish the $\WW^{m,p}$-regularity of solutions of \eqref{eq:gsep} for $\G$ of class $C^{m-1,1}$ under minimal assumptions on the regularity of the coefficients $(\bsfA, \bsfb, \bsfc,\sfd)$.
\end{itemize}

We stress that, in contrast to the typical approach in differential geometry, which assumes $\G$ to be of class $C^\infty$,
we specify the minimal regularity of $\G$ for each statement to hold.
Moreover, since our exposition does not require strong foundations in Riemannian geometry, it may thus serve as  a bridge for works in applied fields.
The results in the present paper are particularly relevant to the numerical analysis community, for which the assumption that the domain is infinitely smooth is overly restrictive and usually unrealistic.
Moreover, building strongly upon the present work we establish in our companion paper \cite{BenavidesNochettoShakipov2025-b} well-posedness and regularity results for a group of prototypical vector-valued hypersurface PDEs, such as the tangent Stokes and Navier--Stokes equations.

A relatively straightforward sanity check consisting of
localizing and rewriting \eqref{eq:gsep} into parametric domain in order to apply the classical PDE theory in flat domains shows that our main results (see \autoref{subsec:main-results} below) are by no means unexpected;
in fact, a well-acquainted analyst might regard them as part of the ``mathematical folklore''.
This reasoning, however, is at least incomplete because it starts from the presumption that a solution to \eqref{eq:gsep} exists, which is not evident for general coefficients $(\bsfA, \bsfb, \bsfc, \sfd)$ and integrability parameter $p \in (1,\infty)$.
Furthermore, despite a comprehensive review of the literature, we were unable to identify any explicit reference for our main results.
Consequently, to the best of our knowledge, this work is the first to provide a unified, self-contained and elementary theory for problem \eqref{eq:gsep} that remains sharp with respect to the regularity of $\G$.
This is the content of \autoref{sec:scal-ell}.

Finally, we note that there exist numerous definitions of Sobolev spaces on manifolds, either in the sense of ``H'' and ``W'' of \cite{H=W} and relying on parametrizations or written in terms of extrinsic calculus, e.g.~\cite[Def.~2.11]{DziukElliott2013}, \cite[\textsection 3]{JerisonKenig1995}, \cite[p.~2]{TaylorIII_2023}, \cite[\textsection 7.1]{Grigoryan2009}, \cite[\textsection 2]{HebeyRobert2008} and \cite[p.~76]{Morrey2008}.
It can be shown, with rather nontrivial arguments, that under suitable assumptions some of these definitions coincide \cite{MathStack-equivalent-Sobolev-manifold}.
We use many of them in this and the companion paper \cite{BenavidesNochettoShakipov2025-b}, depending on the context and regularity of $\G$.
Therefore, in keeping with our emphasis on clarity and self-containment, we include as part of our program a systematic review of Sobolev spaces on hypersurfaces of limited regularity in \autoref{sec:Sobolev-spaces-manifolds}.
\subsection{Preliminary notation}

For $\Omega \subseteq \RR^d$ a bounded Lipschitz domain, we adopt the standard notation for Lebesgue spaces $\LL^q(\Omega)$ and Sobolev spaces $\WW^{m,q}(\Omega)$ and $\WW^{m,q}_0(\Omega)$ with $m \geq 0$ and $q \in [1,\infty)$, whose corresponding norms are denoted respectively by $\|\cdot\|_{0,q;\Omega}$ and $\|\cdot\|_{m,q;\Omega}$.
We also write $\WW^{0,q}(\Omega) = \LL^q(\Omega)$ and $\HH^m(\Omega) := \WW^{m,2}(\Omega)$ with corresponding norm $\|\cdot\|_{m,\Omega} := \|\cdot\|_{m,2;\Omega}$
For any scalar function space $\VV$ defined on $\Omega$, we denote by $\bV$ and $\bbV$ its $d$-vectorial and $d \times d$ tensorial counterparts, and we use the same notation to denote their norms.

We will use a similar notation for the \emph{manifold Sobolev spaces} $\WW^{m,q}(\G)$ defined on $\G$ (these spaces will be properly introduced in \autoref{sec:Sobolev-spaces-manifolds} below), and for the spaces $C^m(\G)$ of $m$-th times continuously differentiable functions on $\G$.
For any scalar function space $\VV$ defined on $\G$, we denote by $\bV$ and $\bbV$ its $(d+1)$-vectorial and $(d+1)\times(d+1)$ tensorial counterparts, and use the same notation for their respective norms.
Moreover, we denote by $\WW^{m,q}_\#(\G)$ the closed subset of scalar-valued functions in $\WW^{m,q}(\G)$ that satisfy $\int_\G v =0$, whereas $\bW^{m,q}_t(\G)$ denotes the space of vector-valued functions $\bv: \G \rightarrow \RR^{d+1}$ that belong to $\bW^{m,q}(\G)$ and are a.e.~tangential to $\G$ (i.e.~$\bv \cdot \bnu = 0$).
Finally, for any normed space $V$, its dual space is denoted by $V'$, whose induced norm is given by $\|f\|_{V'} := \sup_{0 \neq v \in V} \frac{f(v)}{\|v\|_V}$.
We denote the duality pairing between $V'$ and $V$ by $\langle \cdot, \cdot \rangle_{V' \times V}$;
when there is no ambiguity, we will simply write $\langle \cdot,\cdot \rangle$.
Throughout this article, for any integrability parameter $p \in (1,\infty)$, we denote by $p^* := \frac{p}{p-1}$ its conjugate exponent.
Notice that $p^{**} = p$.
Also, $1<p<2$ if and only if $p^*>2$.

\subsection{Background}

The differential operators in \eqref{eq:gsep} satisfy integration-by-parts formulas that highly resemble their flat-domain counterparts \cite[\textsection 2]{DziukElliott2013} (see also \cite[Lemma 1]{BonitoDemlowNochetto2020}).
Said identities naturally induce a set of $\LL^2$-based weak formulations that we now proceed to review.
As far as we know, most of the literature is concerned with the Laplace--Beltrami equation \eqref{eq:lb} rather than the general linear elliptic PDE \eqref{eq:gsep}.

The weak formulation for \eqref{eq:lb}  reads:
given $f \in (\HH^1_\#(\G))'$, find $u \in \HH^1_\#(\G)$ such that
\begin{equation}\label{eq:weak-Laplace--Beltrami-Poisson}
\int_\G \nablaG u \cdot \nablaG v = \langle f , v \rangle, \qquad \forall v \in \HH^1_\#(\G).
\end{equation}
For $\G$ of class $C^{0,1}$ and of any intrinsic dimension $d \geq 2$, the well-posedness of \eqref{eq:weak-Laplace--Beltrami-Poisson} is a consequence of the \emph{Poincaré--Friedrichs} inequality \cite[Lemma 2]{BonitoDemlowNochetto2020} (see also \cite[Theorem 2.12]{DziukElliott2013} for the result for $\G$ of class $C^3$) and the Lax--Milgram theorem.

In flat bounded domains, a-priori estimates in general $\WW^{m,p}$-norms for the scalar Laplace problem are a consequence of the Calderon--Zygmund theory applied to the integral representation of solutions via the Newtonian potential.
For the case $m \geq 2$ we refer to \cite[Chapter 9]{GilbargTrudinger2001} and for the case $m=1$ we refer to \cite[Section 3]{DolzmannMuller1995} (see also and \cite[Section 4, Theorem 1]{Meyers1963} and \cite[Theorem 1.5]{ByunWang2004}).
We exploit this theory in our paper via local parametrizations.
For the vector-valued case \cite{BenavidesNochettoShakipov2025-b}, instead, we adopt a parametrization-free approach that relies on the estimates derived in \autoref{sec:scal-ell} for problems \eqref{eq:gsep} and \eqref{eq:lb}.

Regarding the higher Sobolev regularity of \eqref{eq:lb}, we refer to \cite[Lemma 3]{BonitoDemlowNochetto2020} and \cite[Theorem 3.3]{DziukElliott2013}, which provide\st{s} $\HH^2$-regularity of the solution under the assumption that the right-hand side function $f$ lies in $\LL^2(\G)$ and $\G$ is of class $C^2$.
The proof hinges on converting \eqref{eq:weak-Laplace--Beltrami-Poisson} into a general elliptic PDE over a flat parametric domain and exploiting classical $\HH^2$-interior regularity results, e.g.~\cite[p.~327, Theorem 1]{Evans2010}.

In the context of differential forms, Hodge theory \cite[\textsection 7.4]{Morrey2008} provides $\WW^{2,p}$-regularity ($1 < p < \infty$) for solutions of the Laplace--Beltrami problem in manifolds of class $C^{1,1}$; for $0$-forms, i.e.~scalar functions, the Hodge Laplacian coincides with the Laplace--Beltrami operator.
As far as we are aware, there seem to be no readily available references for regularity results for more general systems of the form \eqref{eq:gsep}.

Motivated by the lack of a comprehensive regularity theory for the general problem \eqref{eq:gsep}, this paper aims to provide a self-contained reference that is sharp on the regularity of $\G$ and data.

\subsection{Statement of the main results}\label{subsec:main-results}
We present our main results concerning well-posedness and higher regularity of solutions to \eqref{eq:gsep}. 
The first two results, \autoref{thm:div-form-wp} and \autoref{thm:div-form-higreg}, establish the well-posedness and higher Sobolev regularity of elliptic problems in divergence form.
These results are the building blocks needed for tackling the general scalar elliptic problems, and they do not follow from the latter because of the zero-mean condition.
Their proofs resort to parametrizations in order to utilize the powerful Calder\'{o}n--Zygmund theory in flat domains, and functional analytical duality arguments.
At the heart of these developments is the observation that the regularity of the Riemannian metric and the coefficient matrix $\bsfA$ intertwine and give rise to the regularity of the coefficient matrix appearing in parametric domain.

\begin{thm}[well-posedness of the operator $-\divG(\bsfA \nablaG \cdot)$ in $\WW^{1,p}_\#(\G)$]
\label{thm:div-form-wp}
Let $p \in (1,\infty)$.
Assume $\G$ is of class $C^1$ and let $\bsfA \in \bbC(\G)$ and that $\bsfA$ is strictly elliptic in $\G$;
that is, there is $\Lambda > 0$ such that
\begin{equation}\label{A-ellipticity-Gamma}
\bxi \cdot \bsfA(\bx) \bxi \geq \Lambda |\bxi|^2,
\end{equation}
for a.e.~$\bx \in \G$ and for each $\bxi \in \RR^{d+1}$ that is tangential to $\G$ at the point $\bx$.

Then for each $f \in (\WW^{1,p^*}_\#(\G))'$, there is a unique $u \in \WW^{1,p}_\#(\G)$ such that
\begin{equation}\label{eq:div-form-W1p-weakform}
\int_\G \bsfA \nablaG u \cdot \nablaG v = \langle f, v \rangle, \qquad \forall v \in \WW^{1,p^*}_\#(\G).
\end{equation}
Moreover, there exists a positive constant $C$ depending only on $\G$, $\bsfA$ and $p$ such that
\begin{equation}\label{eq:div-form-W1p-apriori}
\|u\|_{1,p;\G} \leq C \|f\|_{(\WW^{1,p^*}_\#(\G))'}.
\end{equation}
If $\G$ and $\bsfA$ are respectively only of class $C^{0,1}$ and $\bbL^\infty(\G)$, there exists $\varepsilon>0$, depending only on $\G$ and $\bsfA$, such that \eqref{eq:div-form-W1p-weakform} and \eqref{eq:div-form-W1p-apriori} still hold as long as $p \in (2-\varepsilon,2+\varepsilon)$. 
\end{thm}

\begin{thm}[$\WW^{m+2,p}$-regularity for the operator $-\divG(\bsfA \nablaG \cdot)$ on $\G$]\label{thm:div-form-higreg}
Let $p \in (1,\infty)$.
Assume there is some nonnegative integer $m$ such that $\G$ is of class $C^{m+1,1}$ and let $\bsfA \in \bbC^{m,1}(\G)$.
Moreover, let $\bsfA$ satisfy the ellipticity condition \eqref{A-ellipticity-Gamma}.
If $f \in \WW_\#^{m,p}(\G)$ (which induces an element in $(\WW^{1,p^*}_\#(\G))'$), then the solution $u \in \WW^{1,p}_\#(\G)$ of \eqref{eq:div-form-W1p-weakform} belongs to the space $\WW^{m+2,p}_\#(\G)$
and there exists a constant $C>0$ depending only on $\G$, $\bsfA$, $m$ and $p$ such that
\begin{equation}\label{eq:div-form-higher-apriori}
\|u\|_{m+2,p;\G} \leq C \|f\|_{m,p;\G}.
\end{equation}
Moreover, the solution $u \in \WW^{m+2,p}_\#(\G)$ satisfies
\begin{equation}\label{div-form-a.e.}
-\divG(\bsfA\nablaG u) = f, \qquad \text{a.e.~on $\G$}.
\end{equation}
\end{thm}
We build on \autoref{thm:div-form-wp} and \autoref{thm:div-form-higreg} to also establish the well-posedness (cf.~\autoref{thm:gsep-wp}) and higher regularity (cf.~\autoref{thm:gsep-reg}) for weak solutions of the general scalar elliptic equation \eqref{eq:gsep} by means of the Fredholm alternative.
\begin{thm}[well-posedness of general scalar elliptic equations] \label{thm:gsep-wp}
Let $p \in (1,\infty)$.
Assume $\G$ is of class $C^1$, and consider coefficients $(\bsfA, \bsfb, \bsfc, \sfd) \in \bbC(\G) \times \bL^\infty_t(\G) \times \bL^\infty_t(\G) \times \LL^\infty(\G)$, with $\bsfA$ satisfying the ellipticity condition \eqref{A-ellipticity-Gamma}.
Moreover, assume that there exists a constant $\lambda > 0$ and a subset $M$ of $\G$ with positive measure such that at least one of the following conditions holds
\begin{subequations}
\begin{align}
\label{cond-coeffs-with-b}\int_\G (\sfd\,w + \bsfb \cdot \nablaG w) & \geq \lambda \int_M w, \qquad \forall w \in \WW^{1,1}(\G), \text{$w \geq 0$ on $\G$},\\
\label{cond-coeffs-with-c}\int_\G (\sfd\,w + \bsfc \cdot \nablaG w) & \geq \lambda \int_M w, \qquad \forall w \in \WW^{1,1}(\G), \text{$w \geq 0$ on $\G$}.
\end{align}
\end{subequations}
Then, for every $f \in (\WW^{1,p^*}(\G))'$, there exists a unique $u \in \WW^{1,p}(\G)$ such that
\begin{equation}\label{eq:gsep-weakform}
\int_\G \bsfA \nablaG u \cdot \nablaG v + u (\bsfb \cdot \nablaG v)  + (\bsfc \cdot \nablaG u) v + \sfd \, u v  = \langle f, v \rangle, \qquad \forall v \in \WW^{1,p^*}(\G).
\end{equation}
Moreover, there exists a constant $C>0$, depending only on $\G$, $\bsfA$, $\bsfb$, $\bsfc$, $\sfd$ and $p$, such that
\begin{equation}\label{eq:gsep-apriori}
\|u\|_{1,p;\G} \leq C \|f\|_{(\WW^{1,p^*}(\G))'}.
\end{equation}
This result is still valid for $\G$ of class $C^{0,1}$ and $\bsfA \in \bbL^\infty(\G)$ provided $p \in (2-\varepsilon, 2+\varepsilon)$ for $\varepsilon > 0$ sufficiently small depending only on $\G$ and $\bsfA$.
\end{thm}

The requirement of positivity of the constant $\lambda$ in \eqref{cond-coeffs-with-b} and \eqref{cond-coeffs-with-c} may seem unusual at first, since in flat bounded domains with zero Dirichlet boundary conditions $\lambda$ can be chosen to be zero \cite[eq.~(8.8)]{GilbargTrudinger2001}.
This is due to the fact that $\G$ does not have a boundary, which implies that $\|\nablaG \cdot\|_{0,\G}$ does not control the full $\HH^1(\G)$-norm, and hence $\lambda>0$ is required for a Poincaré-type inequality to hold;
we refer to the proof of \autoref{prop:uniqueness-GSES} for further details.

In the theorem below, solely for simplicity of notation, we set: $\bC^{-1,1}_t(\G) := \bL^\infty_t(\G)$ and $\CC^{-1,1}(\G) := \LL^\infty(\G)$.

\begin{thm}[higher regularity for general scalar elliptic equations]\label{thm:gsep-reg}
Let $p \in (1,\infty)$.
Assume there is some nonnegative integer $m$ such that $\G$ is class $C^{m+1,1}$ and let $(\bsfA, \bsfb, \bsfc, \sfd) \in \bbC^{m,1}(\G) \times \bC_t^{m,1}(\G) \times \bC^{m-1,1}_t(\G) \times C^{m-1,1}(\G)$, with $\bsfA$ satisfying condition \eqref{A-ellipticity-Gamma}.
Moreover, assume that at least one of the conditions \eqref{cond-coeffs-with-b} and \eqref{cond-coeffs-with-c} is verified.
Then, for each $f \in \WW^{m,p}(\G)$ the solution $u \in \WW^{1,p}(\G)$ of \eqref{eq:gsep-weakform} is in fact in $\WW^{m+2, p}(\G)$
and satisfies
\begin{equation}\label{gsep-a.e.}
- \divG (\bsfA \nablaG u + u\, \bsfb ) + \bsfc \cdot \nablaG u + \sfd u = f, \qquad \text{a.e.~on $\G$}.
\end{equation}
Moreover, there is a constant $C >0$, depending only on $\G$, $\bsfA$, $\bsfb$, $\bsfc$, $\sfd$ and $p$, such that
\begin{equation}\label{eq:gsep-higher-apriori}
\|u\|_{m+2,p;\G} \leq C \|f\|_{m,p;\G}.
\end{equation}
\end{thm}

Even though all of our results are stated for compact manifolds of class $C^m$ (or $C^{m-1,1}$) without boundary, upon appropriate modifications, we expect them to extend to the case of a $d$-dimensional manifold $\G$ of the same regularity class with a sufficiently regular boundary $\partial \G$.
In particular, when expressing the problem in the parametric domain, boundary regularity estimates would need to be resorted to in addition to the interior regularity estimates employed in this paper.

The rest of the paper is structured as follows.
We start in \autoref{sec:diff-geo} with basic concepts from differential geometry, transitioning into \autoref{sec:Sobolev-spaces-manifolds} with a careful examination of Sobolev spaces on manifolds.
\autoref{sec:scal-ell} is devoted to proving the well-posedness and higher regularity of \eqref{eq:gsep}.
In \autoref{subsec:lp-reg-divform} we show \autoref{thm:div-form-wp} and \autoref{thm:div-form-higreg} for PDEs in divergence form whereas in \autoref{subsec:lp-reg-GSES} we prove \autoref{thm:gsep-wp} and \autoref{thm:gsep-reg} for general scalar elliptic PDEs.
We deal with extensions to convection-diffusion equations in \autoref{subsec:extensions} and surface biharmonic equations in \autoref{subsec:biharmonic}.
We conclude with \autoref{app:var-prob} (Variational problems in reflexive Banach spaces) and \autoref{app:regularity-flat} (Regularity estimates in flat domains), which make the paper self-contained.
\section{Preliminaries on differential geometry}\label{sec:diff-geo}

In this section, we collect classical results on calculus on manifolds;
we refer to Dziuk and Elliott \cite[\textsection 2]{DziukElliott2013}, Bonito, Demlow, and Nochetto \cite[\textsection 1.2]{BonitoDemlowNochetto2020} and Shahshahani \cite{Shahshahani2016}.
We primarily focus on the definitions of geometric quantities and differential operators on manifolds that are less regular than $C^2$---given the lack of regularity, such definitions have to involve parameterizations.

Let $\G$ be a $d$-dimensional ($d \geq 2$) compact and connected manifold without boundary embedded in $\RR^{d+1}$ of class $C^{0,1}$ (see \cite[p.~102]{Shahshahani2016} for the definition of a manifold of class $C^m$ for $m \in \NN$, which naturally extends to $C^{m-1,1}$, and \cite[p.~134]{Shahshahani2016} for the definition of an embedded manifold).
We assume connectedness out of convenience, otherwise we could proceed independently for each connected component.
Since $\G$ is compact, there exists a finite atlas $\{(\cV_i,\cU_i,\wt\bchi_i)\}_{i=1}^N$ of $\G$, where each of the charts $\wt\bchi_i: \cV_i \rightarrow \cU_i \cap \G$ are isomorphisms of class $C^{0,1}$ compatible with the orientation of $\G$.
The sets $\cV_i$ are open connected subsets of $\RR^d$ and the sets $\cU_i$ are open subsets of $\RR^{d+1}$.
Without loss of generality we can assume that there exist domains $\cW_i \subseteq \RR^{d+1}$ such that $\overline{\cW_i} \subseteq \cU_i$ and $\{\cW_i\}_{i=1}^N$ is still a covering of $\G$ \cite[Ch.~4, Lemma 10]{Shahshahani2016}.

By using a partition of unity, in general it will be enough to consider a single chart;
hence, we drop the index $i$ for ease of notation whenever no ambiguity arises.
For $\bx \in \cU \cap \G$, we write $\by:= \wt\bchi^{-1}(\bx) \in \cV$.
Moreover, for every (scalar, vector or tensor) function $\bf$ defined on $\cU \cap \G$ we let $\wt\bf := \bf \circ \wt\bchi$, which is defined on $\cV$;
equivalently $\bf(\bx) = \wt\bf(\by)$ for each $\by \in \cV$ and $\bx =\wt\bchi(\by)$.

The \emph{first fundamental form} is given by $\wt\bg(\by) = \nabla^T \wt\bchi(\by) \nabla \wt\bchi(\by) \in \RR^{d \times d}$ for each $\by \in \cV$.
The non-degeneracy of $\wt\bchi$ on $\cV$ ensures that $\wt\bg = (\wt g_{ij})_{i,j=1}^d$ is symmetric uniformly positive definite on $\cV$.
Its inverse matrix is denoted by $\wt\bg^{-1} = (\wt g^{ij})_{i,j=1}^d$.

The \emph{area element} $\wt a_{\bchi}$ is defined by $\wt a_{\bchi}(\by) = \sqrt{\det \wt\bg(\by)}$.
A measurable function $v: \cU \cap \G \rightarrow \RR$ is integrable if and only if the measurable function $\wt v: \cV \rightarrow \RR$ is integrable.
The area element links their integral values via the identity
\begin{equation}\label{change-of-variable}
\int_{\cU \cap \G} v = \int_\cV \wt v \, \wt a_{\bchi}.
\end{equation}
An a.e.~defined unit normal vector field $\bnu = (\nu_i)_{i=1}^{d+1}: \G \rightarrow \RR^{d+1}$ to $\G$ is given by $\bnu := \bN/|\bN|$, where $\wt\bN(\by) := \sum_{j=1}^{d+1} \det(\be_j, \nabla \wt\bchi(\by)) \be_j$ for $\by \in \cV$; here $\{\be_j\}_{j=1}^{d+1}$ is the canonical basis of $\RR^{d+1}$.
We denote by $\bP := \bI - \bnu \otimes \bnu$ the projection operator onto the tangent plane to $\G$, which can be expressed in terms of the parametrization as
\begin{equation}\label{eq:tang-proj}
\bP(\bx) =  \wt\bP(\by) = \nabla\wt\bchi(\by) \wt\bg^{-1}(\by) \nabla^T\wt\bchi(\by).
\end{equation}

The \emph{exterior gradient} $\nablaM$ acts on $C^{0,1}$ scalar fields $v: \G \rightarrow \RR$ and vector fields $\bv = (v_i)_{i=1}^{d+1} : \G \rightarrow \RR^{d+1}$, and is given by
\begin{equation}\label{surface-gradient-param}
\nablaM v(\bx) := \nabla \wt\bchi(\by) \wt\bg(\by)^{-1} \nabla \wt v(\by), \qquad\qquad
\nablaM \bv := \begin{bmatrix} \nablaM^T v_1 \\ \vdots \\ \nablaM^T v_{d+1}\end{bmatrix};
\end{equation}
notice that $\bP \nablaM v = \nablaM v$.
The \textit{exterior divergence} $\divM$ acts on $C^{0,1}$ vector fields $\bv: \G \rightarrow \RR^{d+1}$ and tensor fields $\bA: \G \rightarrow \RR^{(d+1)\times(d+1)}$, and is defined by
\begin{equation}\label{surface-divergence-param}
\divM \bv(\bx) := \tr(\nablaM \bv(\bx)) = \sum_{i,j=1}^d \wt g^{ij}(\by) \partial_i \wt\bchi(\by) \cdot \partial_j \wt\bv(\by), \qquad
\divM \bA = (\divM \bA_{i,\cdot} )_{i=1}^{d+1},
\end{equation}
where $\bA_{i,\cdot}$ denotes the $i$-th row of $\bA$.

On the other hand, the \emph{covariant gradient} $\nablaG$ is defined by
\begin{equation}\label{covariant-surface-gradient}
\nablaG v = \nablaM v, \qquad
\nablaG \bv = \bP \nablaM \bv
= \bP \begin{bmatrix} \nablaG^T v_1 \\ \vdots \\ \nablaG^T v_{d+1}\end{bmatrix}.
\end{equation}
The \textit{covariant divergence} $\divG$ acts on vector fields $\bv: \G \rightarrow \RR^{d+1}$ and tensor fields $\bA: \G \rightarrow \RR^{(d+1)\times(d+1)}$, and is defined by
\begin{equation}\label{covariant-surface-divergence}
\divG \bv = \tr(\nablaG \bv), \qquad
\divG \bA = (\divG \bA_{i,\cdot} )_{i=1}^{d+1}.
\end{equation}
It turns out that the exterior and covariant divergences coincide, i.e.,
\begin{equation*}
\divM \bv = \divG \bv, \qquad
\divM \bA = \divG \bA.
\end{equation*}
If $\G$ and $v:\G \rightarrow \RR^{d+1}$ are of class $C^{1,1}$, then we define the Laplace--Beltrami operator $\DeltaG$ by
\begin{equation}\label{Laplace--Beltrami-param}
\DeltaG v(\bx) := \divG \nablaG v(\bx) = \frac{1}{\wt a_{\bchi}(\by)} \div(\wt a_{\bchi}(\by) \bg(\by)^{-1} \nabla \wt v(\by)).
\end{equation}
It is sometimes convenient to write $\nablaG v = (\uD_1 v, \dotsc, \uD_{d+1} v)$ for the $d+1$ components of $\nablaG v = \nablaM v$;
that is
\begin{equation}\label{surface-tangentialderivative-param}
\uD_i v(\bx) = \sum_{k,l=1}^{d+1} \partial_k \wt\chi_i(\by) \wt g^{kl}(\by) \partial_l \wt v(\by).
\end{equation}
This notation allows us, in particular, to express the action of $\divG$ and $\DeltaG$ in compact form:
\begin{equation*}
\divG \bv = \sum_{k=1}^{d+1} \uD_k v_k, \qquad \qquad \DeltaG v = \sum_{k=1}^{d+1} \uD_k \uD_k v.
\end{equation*}
For $\G$ of class $C^{1,1}$, we write $\bB := \nablaM \bnu \in \bbL^\infty(\G)$ for the so-called \emph{shape operator} of $\G$ (aka \emph{Weingarten map}).
Since $|\bnu|^2 = 1$, it follows that $\bB \bnu = \mathbf{0}$ a.e.~on $\G$.
Notice that for $\G$ of class $C^2$, it can be proved that $\bB$ coincides with the Hessian matrix of the distance function to $\G$, and so it is symmetric.

\begin{remark}[independence of parametrizations]\label{rem:diff-op-ind-param}
It is possible to show that for $\G$ of class $C^{0,1}$, the definitions \eqref{eq:tang-proj} and \eqref{surface-gradient-param} of $\bP(\bx)$ and $\nabla_M v(\bx)$ are independent of the choice of parametrizations: if two parametrizations $\wt\bchi_i: \cV_i \rightarrow \cU_i \cap \G$, $i \in \{1,2\}$ are such that $\bx \in \cU_1 \cap \cU_2 \cap \G$, then, by denoting $\by_i = \wt\bchi_i^{-1}(\bx)$, $\wt\bg_i = \nabla^T \wt\bchi_i \nabla \wt\bchi_i$ and $\wt v_i = v \circ \wt\bchi_i$, there holds that
\begin{align*}
\nabla\wt\bchi_1(\by_1) \wt\bg_1^{-1}(\by_1) \nabla^T\wt\bchi_1(\by_1)
& = \nabla\wt\bchi_2(\by_2) \wt\bg_2^{-1}(\by_2) \nabla^T\wt\bchi_2(\by_2),\\
\nabla\wt\bchi_1(\by_1) \wt\bg_1^{-1}(\by_1) \nabla\wt v_1(\by_1)
& = \nabla\wt\bchi_2(\by_2) \wt\bg_2^{-1}(\by_2) \nabla\wt v_2(\by_2).
\end{align*}
The proof follows from straightforward applications of the chain rule applied to identities $\wt\bchi_2 = \wt\bchi_1 \circ (\wt\bchi_1^{-1} \circ \wt\bchi_2)$ and $\wt v_1 \circ (\wt\bchi_1^{-1} \circ \wt\bchi_2) = \wt v_2$, and simple algebraic manipulations.
We omit further details.

Consequently, all differential operators on $\G$ introduced throughout the present section are also independent of the chosen parametrizations.
\end{remark}

We finish this section by mentioning that for $\G$ is of class $C^2$, we can express $\nablaM$ in a purely extrinsic (that is, independent of parametrizations) manner:
\begin{equation*}
\nablaM v = \bP \left(\nabla v^e\right)\rvert_\G, \qquad \nablaM \bv = \left(\nabla \bv^e\right)\rvert_\G \bP,
\end{equation*}
where $v^e$ (resp.~$\bv^e$) denotes an
extension \cite[eq.~(1.62)]{BonitoDemlowNochetto2020} of $v$ (resp.~$\bv$) to an $(d+1)$-dimensional tubular neighborhood $\Omega_\delta$ of $\G$ for which the distance function to $\G$, $\bx \mapsto \operatorname{dist}(\bx,\G)$ for $\bx \in \Omega_\delta$, remains of class $C^2$ \cite[Lemma~2.8]{DziukElliott2013}.

\section{Sobolev spaces on \texorpdfstring{$\G$}{Γ}}\label{sec:Sobolev-spaces-manifolds}

In this section, we carefully and systematically review several definitions of Sobolev spaces on manifolds, with particular emphasis on the regularity of $\G$ for which these definitions coincide.

\subsection{Definition of Sobolev spaces on \texorpdfstring{$\G$}{Γ}}
We start with Sobolev spaces over manifolds of limited regularity (\autoref{def:Sobolev-spaces-manifold}) in terms of parameterizations, and next show that this definition is independent of the latter (\autoref{prop:sob-ind-param}).
Consequently, we can extend the definitions of differential operators to Sobolev functions (\autoref{def:diff-ops-SSS}), derive crucial density results (\autoref{thm:density-C^m}), and finally introduce an equivalent, but more convenient norm on Sobolev spaces that is parameterization-free (\autoref{prop:param_ind-Sobolev-norm}).

For our purposes we find it convenient to begin by adopting definition \cite[p.~76]{Morrey2008} for a compact manifold $\G$.

\begin{definition}[Sobolev spaces on $\G$]\label{def:Sobolev-spaces-manifold}
Let $\G$ be of class $C^{m-1,1}$, for some $m \in \NN$, and let $\{(\cV_i,\cU_i,\wt\bchi_i)\}_{i=1}^N$ be two finite atlases of $\G$ (cf.~\autoref{sec:diff-geo}).
For $p \in [1,\infty)$, we define the Sobolev space $\WW^{m,p}(\G)$ by
\begin{equation*}
\WW^{m,p}(\G) := \{ u:\Gamma \rightarrow \RR \mid u \circ \wt\bchi_i \in \WW^{m,p}(\cV_i), \quad \forall i=1,\dotsc,N\},
\end{equation*}
endowed with the norm
\begin{equation}\label{scalar-Sobolev-norm}
\|u\|_{m,p;\G} := \left( \sum_{i=1}^N \|u \circ \wt\bchi_i\|_{m,p;\cV_i}^p \right)^{1/p}.
\end{equation}
\end{definition}

The following result shows that the Sobolev space $\WW^{m,p}(\G)$ as defined above is independent of the chosen atlas.
It was mentioned in \cite[p.~76]{Morrey2008} albeit without proof, which we now provide for completeness.

\begin{proposition}[independence of parametrizations]\label{prop:sob-ind-param}
The definition of $\WW^{m,p}(\G)$ is independent of the chosen atlas.
Moreover, given two finite atlases of $\G$, their induced norms \eqref{scalar-Sobolev-norm} are topologically equivalent.
\begin{proof}
Let $\{(\cV_i,\cU_i,\wt\bchi_i)\}_{i=1}^N$ and $\{(\sV_j,\sU_j,\wt\bxi_j)\}_{j=1}^M$ be two finite atlases of $\G$ of class $C^{m-1,1}$.
Let $u:\G \rightarrow \RR$ satisfy $u \circ \wt\bchi_i \in \WW^{m,p}(\cV_i)$ for each $i=1\dotsc,N$.
Without loss of generality we will assume that $\cU_i \cap \sU_j \cap \G \neq \varnothing$.

Let $j \in \{1,...,M\}$ be fixed. Since $u \circ \wt\bchi_i \in \WW^{m,p}(\cV_i)$ and $\wt\bchi_i^{-1} \circ \wt\bxi_j$ and its inverse are of class $C^{m-1,1}$, we observe that equality $u \circ \wt\bxi_j = (u \circ \wt\bchi_i) \circ (\wt\bchi_i^{-1} \circ \wt\bxi_j)$ holds a.e.~in $\wt\bxi_j^{-1}(\cU_i \cap \sU_j \cap \G)$. Thus, by ``changing coordinates'' \cite[Theorem 3.1.7]{Morrey2008}, \cite[Theorem 11.54]{Leoni:2009} (see also \cite[Exercise 4.2.11]{Weber2018-lecturenotes}), it follows that $(u \circ \wt\bxi_j)\rvert_{\wt\bxi_j^{-1}(\cU_i \cap \sU_j \cap \G)} \in \WW^{m,p}(\wt\bxi_j^{-1}(\cU_i \cap \sU_j \cap \G))$ and $\|u \circ \wt\bxi_j\|_{m,p;\wt\bxi_j^{-1}(\cU_i \cap \sU_j \cap \G)} \lesssim \|u \circ \wt\bchi_i\|_{m,p;\wt\bchi_i^{-1}(\cU_i \cap \sU_j \cap \G)}$.
Thus, noticing that $\bigcup_{i=1}^N \wt\bxi_j^{-1}(\cU_i \cap \sU_j \cap \G) = \wt\bxi_j^{-1} \left( \bigcup_{i=1}^N \cU_i \cap \sU_j \cap \G \right) = \wt\bxi_j^{-1}(\sU_j \cap \G) = \sV_j$, we deduce that $u \circ \wt\bxi_j \in \WW^{m,p}(\sV_j)$ and, by the finite overlapping property of the underlying sets, that
\begin{equation*}
\|u \circ \wt\bxi_j\|_{m,p;\sV_j}
\lesssim \sum_{i=1}^N \|u \circ \wt\bxi_j\|_{m,p;\wt\bxi_j^{-1}(\cU_i \cap \sU_j \cap \G)}
\lesssim \sum_{i=1}^N \|u \circ \wt\bchi_i\|_{m,p;\wt\bchi_i^{-1}(\cU_i \cap \sU_j \cap \G)}
\leq \sum_{i=1}^N \|u \circ \wt\bchi_i\|_{m,p;\cV_i}.
\end{equation*}
The proof is concluded after summing over $j=1,\dotsc,M$.
\end{proof}
\end{proposition}

It is easy to see that $C^{m-1,1}(\G) \subseteq \WW^{m,p}(\G)$.
We write $\bW^{m,p}(\G) := [\WW^{m,p}(\G)]^{d+1}$ and endow it with its natural product-space norm.
For $p = 2$, we also write $\bH^{m}(\G) := \bW^{m,2}(\G)$.

We are now in position to define differential operators on $\WW^{m,p}(\G)$.

\begin{definition}[differential operators on Sobolev spaces on $\G$]\label{def:diff-ops-SSS}
For $\G$ of class $C^{0,1}$ and $p \in [1,\infty)$, we naturally extend the definition of the linear operators $\nablaM$, $\nablaG$, $\div_M$ and $\divG$ acting on $C^{1,1}(\G)$ to $\WW^{1,p}(\G)$ (and its vector- and tensor-valued versions) by the a.e.~valid formulas \eqref{surface-gradient-param}, \eqref{surface-divergence-param}, \eqref{covariant-surface-gradient} and \eqref{covariant-surface-divergence}.

Similarly, if $\G$ is of class $C^{m-1,1}(\G)$ for some $m \in \NN$, the appropriate composition of the aforementioned operators allows us to define differential operators of order $m$ acting on elements of $\WW^{m,p}(\G)$.

On account of \autoref{rem:diff-op-ind-param}, all these operators are irrespective of the chosen finite atlas of $\G$.
\end{definition}

Finally, we introduce the following spaces
\begin{align*}
\LL^p_\#(\G) & := \{\textstyle  q \in \LL^p(\G): \int_\G q = 0 \},\\
\WW^{m,p}_\#(\G) & := \LL^p_\#(\G) \cap \WW^{m,p}(\G),\\
\bW^{m,p}_t(\G) & := \{ \bv \in \bW^{m,p}(\G): \bv \cdot \bnu = 0~\text{a.e.~on $\G$} \},
\end{align*}
which are respectively closed subspaces of $\LL^p(\G)$, $\WW^{m,p}(\G)$ and $\bW^{m,p}(\G)$.
Notice that, if $\G$ is of class $C^{m,1}$ for some nonnegative integer $m$, $\WW^{m+1,p}(\G) \subseteq \WW^{m,p}(\G)$ and $\nablaM u \in \bW^{m,p}_t(\G)$ for each $u \in \WW^{m+1,p}(\G)$.

The following result states that Sobolev spaces on $\G$ behave well under multiplication by sufficiently smooth functions.
Its a proof is a straightforward consequence of the classical Leibniz rule in flat domains.
\begin{lemma}[product rule on $\WW^{m,p}(\G)$]\label{lem:product-by-smooth}
Assume $\G$ is of class $C^{m-1,1}$ for some $m \in \NN$, and let $p \in [1,\infty)$.
Then for each $\theta \in C^{m-1,1}(\G)$ and $u \in \WW^{m,p}(\G)$, the product $\theta u$ belongs to $\WW^{m,p}(\G)$,
\begin{equation}\label{eq:product-by-smooth}
\nablaM (\theta u) = \theta \nablaM u + u \nablaM \theta, \quad \text{a.e.~on $\G$}.
\end{equation}
and there exists a constant $C_\theta>0$, depending only on $\theta$, such that
\begin{equation}\label{eq:stab_product-by-smooth}
\|\theta u\|_{m,p;\G} \leq C_\theta \|u\|_{m,p;\G}.
\end{equation}
\end{lemma}
The following remark is a direct consequence of \autoref{def:Sobolev-spaces-manifold} for functions that are supported on one chart, which will be used to prove a density result (cf.~\autoref{thm:density-C^m}).
\begin{remark}[norm equivalence]\label{rem:equivalence-parametric-surface}
Let $p \in [1,\infty)$.
According to the the notation from \autoref{sec:diff-geo},
if $u \in \WW^{m,p}(\G)$ is supported on $\cW \cap \G$, then $\wt u \in \WW^{m,p}_0(\cV)$ and there exist positive constants $C_1$ and $C_2$ depending only on $\G$ and $m$ such that
\begin{equation}\label{eq:norm-equivalence}
C_1 \|u\|_{m,p;\G} \leq \|\wt u\|_{m,p;\cV} \leq C_2 \|u\|_{m,p;\G}.
\end{equation}
\end{remark}
\begin{thm}[density]\label{thm:density-C^m}
Let $p \in [1,\infty)$ and $m \in \NN$.
If $\G$ is of class $C^{m-1,1}$ (resp.~$C^m$), then $C^{m-1,1}(\G)$ (resp.~$C^m(\G)$) is dense in $\WW^{m,p}(\G)$.
\begin{proof}
Recall the notation introduced in \autoref{sec:diff-geo}.
Let $u \in \WW^{m,p}(\G)$.
Let $\{\theta_i\}_{i=1}^N$ be a $C^{m-1,1}$ (resp.~$C^m$) partition of unity associated with the covering $\{\cW_i\}_{i=1}^N$ of $\G$ \cite[Ch.~4, 11.~Theorem]{Shahshahani2016}.
By \autoref{lem:product-by-smooth} and \autoref{rem:equivalence-parametric-surface}, we know that $\wt\mho_i \in \WW^{m,p}_0(\cV_i)$, where $\mho_i := \theta_i u$ is supported on $\cW_i \cap \G$ and belongs to $\WW^{m,p}(\G)$.
Then, given $\varepsilon>0$, there exists $\wt\varphi_i \in C^\infty_c(\cV_i)$ such that $\|\wt\mho_i - \wt\varphi_i\|_{m,p;\cV_i} < \varepsilon/N$.
Extend each $\varphi_i: \Gamma \rightarrow \RR$ by $0$ on $\G$ and define $\varphi := \sum_{i=1}^N \varphi_i \in C^{m-1,1}(\G)$ (resp.~$C^m(\G)$).
Then, using the triangle inequality and recalling \autoref{rem:equivalence-parametric-surface} yields
\begin{equation*}
\|u - \varphi\|_{m,p;\G} \leq \sum_{i=1}^N \|\mho_i - \varphi_i\|_{m,p;\G} \cong \sum_{i=1}^N \|\wt\mho_i - \wt\varphi_i\|_{m,p;\cV_i} < \varepsilon,
\end{equation*}
which finishes the proof.
\end{proof}
\end{thm}
It is helpful to endow $\WW^{m,p}(\G)$ with an equivalent norm that is parametrization-independent and resembles its flat-domain counterpart.
This is the content of the following result.

\begin{proposition}[equivalent norm in $\WW^{m,p}(\G)$]\label{prop:param_ind-Sobolev-norm}
Assume $\G$ is of class $C^{m-1,1}$, for some $m \in \NN$, and let $p \in [1,\infty)$.
Then, the mapping
\begin{equation}\label{param_ind-scalar-Sobolev-norm}
u \mapsto \left(\sum_{k=0}^m \|\nablaM^k u \|_{0,p;\G} ^p\right)^{1/p},
\end{equation}
defines an equivalent norm in $\WW^{m,p}(\G)$.
Here, $\nablaM^k u$ denotes the \emph{$k$-fold weak extrinsic gradient} of $u$, which is the $k$-dimensional tensor defined by $(\nablaM^k u)_\bj = \uD_{j_{k}} \cdots \uD_{j_1} u$ for each $\bj = (j_i)_{i=1}^k \in \{1,\dotsc,d+1\}^k$;
we also write $\nablaM^0 u := u$.
\begin{proof}
Denote \eqref{param_ind-scalar-Sobolev-norm} by $||| \cdot |||_{m,p;\G}$.
We proceed in two steps.

\noindent\textbf{Step 1: Compact support.}
We first prove the result for functions supported on $\cW \cap \G$.
Let $u \in \WW^{m,p}(\G)$ supported on $\cW \cap \G$.
Then, by \autoref{rem:equivalence-parametric-surface}, $\wt u \in \WW^{m,p}_0(\cV)$ and $\|\wt u\|_{m,p;\cV} \cong \|u\|_{m,p;\G}$.
Recall from \eqref{surface-gradient-param} that $\wt{\nablaM u} = \nabla \wt\bchi \wt\bg^{-1} \nabla \wt u$ in $\cV$, or equivalently, $\nabla \wt u = (\nabla^T \wt\bchi)\wt{\nablaM u}$.
Hence, since $\wt\bchi$ is of class $C^{m-1,1}$ and $\bg^{-1}$ is of class $C^{m-2,1}$ (or $\LL^\infty$ if $m = 1$), it follows from the Leibniz rule for Sobolev spaces in flat domains (see \cite[\textsection 5.2.3-(iv)]{Evans2010}) that 
\begin{equation*}
\|\wt{\nablaM u}\|_{m-1,p;\cV} \cong \|\nabla \wt u\|_{m-1,p;\cV}.
\end{equation*}
Hence, by ``a change of variables'' (cf.~\cite[Theorem 3.1.7]{Morrey2008}, \cite[Exercise 4.2.11]{Weber2018-lecturenotes}), we deduce that
\begin{equation*}
\|u\|_{m,p;\G}
\cong \|\wt u\|_{m,p;\cV}
\cong \|\wt u\|_{0,p;\cV} + \|\nabla \wt u\|_{m-1,p;\cV}
\cong \|\wt u\|_{0,p;\cV} + \|\wt{\nablaM u}\|_{m-1,p;\cV}
\cong \|u\|_{0,p;\G} + \|\nablaM u\|_{m-1,p;\G}.
\end{equation*}
Since $\nablaM u \in \WW^{m-1,p}(\G)$ (i.e.,~$\uD_i u \in \WW^{m-1,p}(\G)$ for each $i=1,\dotsc,d+1$), we can proceed recursively to finally obtain
\begin{equation*}
\|u\|_{m,p;\G} \cong \sum_{k=0}^m \|\nablaM^k u \|_{0,p;\G} \cong |||u|||_{m,p;\G},
\end{equation*}
which finishes the proof of Step 1.

\noindent\textbf{Step 2: General case.}
As in the proof of \autoref{thm:density-C^m}, let $\mho_i := \theta_i u \in \WW^{m,p}(\G)$ be the ``localization'' of $u$ to $\cW_i \cap \G$.
Using the triangle inequality and applying Step 1 to each $\mho_i$, it follows that
\begin{equation*}
\|u\|_{m,p;\G}
\leq \sum_{i=1}^N \|\mho_i\|_{m,p;\G}
\cong \sum_{i=1}^N |||\mho_i|||_{m,p;\G}
\lesssim |||u|||_{m,p;\G},
\end{equation*}
where in the last inequality we have utilized identity \eqref{eq:product-by-smooth} of \autoref{lem:product-by-smooth} $m$-times and bound the first $m$ exterior derivatives of $\theta$.
Similarly, we have
\begin{equation*}
|||u|||_{m,p;\G}
\leq \sum_{i=1}^N |||\mho_i|||_{m,p;\G}
\cong \sum_{i=1}^N \|\mho_i\|_{m,p;\G}
\lesssim \|u\|_{m,p;\G},
\end{equation*}
where the last inequality is due to estimate \eqref{eq:stab_product-by-smooth}.
This finishes the proof.
\end{proof}
\end{proposition}
Throughout the rest of the paper, it will be usually convenient to utilize \eqref{param_ind-scalar-Sobolev-norm} as the ``main'' norm on $\WW^{m,p}(\G)$ and consequently we will also denote it by $\|\cdot\|_{m,p;\G}$ when no confusion arises.
Notice that this equivalent norm naturally induces a notion of Sobolev norms on subsets of $\G$:
given $\gamma \subseteq \G$, a measurable set with respect to the Lebesgue measure on $\G$, we write
\begin{equation*}
\|v\|_{m,p;\gamma} := \left(\sum_{k=0}^m \|1_\gamma \, \nablaM^k v \|_{0,p;\G} ^p\right)^{1/p}, \qquad \forall v \in \WW^{m,p}(\G),
\end{equation*}
where $1_\gamma$ denotes the characteristic function of $\G$.

\subsection{Comparison between different definitions}
We now comment on the comparison between our adopted definition \autoref{def:Sobolev-spaces-manifold} of Sobolev spaces and a few others found in the literature.

\begin{remark}[comparison with $H^p_m(\G)$ of Hebey and Robert \cite{HebeyRobert2008}]
\autoref{thm:density-C^m} and \autoref{prop:param_ind-Sobolev-norm} show that $\WW^{m,p}(\G)$ (cf.~\autoref{def:Sobolev-spaces-manifold}) coincides with the space $H^p_m(\G)$ of \cite[Definition 2.1, Proposition 2.2]{HebeyRobert2008}.
We stress that although the space $H^p_m(\G)$ of \cite{HebeyRobert2008} was introduced for $\G$ of class $C^\infty$, it can be easily seen that its definition is still valid for manifolds of appropriate limited regularity.
\end{remark}
From \cite[Lemma B.5]{BouckNochettoYushutin2024} (see also \cite[Proposition 12]{BonitoDemlowNochetto2020}) and a density argument (cf.~\autoref{thm:density-C^m} and \autoref{prop:param_ind-Sobolev-norm}) we know that the integration-by-parts formula:
for each $u \in \WW^{1,1}(\G)$ and $\bvarphi \in \bC^1(\G)$
\begin{equation}\label{eq:int-by-parts}
\int_\Gamma u  \, \divG \bvarphi = - \int_\Gamma \nablaG u \cdot \bvarphi + \int_\Gamma \tr(\bB) u \bvarphi \cdot \bnu,
\end{equation}
holds for $\G$ of class $C^2$.
This shows that for $\G$ of class $C^2$, \autoref{def:Sobolev-spaces-manifold} coincides with Dziuk and Elliott \cite[Definition 2.11]{DziukElliott2013}.
For $\G$ of class $C^{1,1}$ the argument is more delicate;
\cite[Lemma B.5]{BouckNochettoYushutin2024} and \cite[Proposition 12]{BonitoDemlowNochetto2020} are proved by means of extrinsic calculus manipulations that hinge on the fact that the distance function $\sfd$ to $\G$ is of class $C^2$ near $\G$ if $\G$ is of class $C^2$ \cite[Theorem 1]{Foote1984}, whereas if $\G$ is only of class $C^{1,1}$ the authors are unaware of any readily available reference that states that $\sfd$ inherits such regularity near $\G$.
With these considerations in mind, we proceed to prove \eqref{eq:int-by-parts} for $\G$ of class $C^{1,1}$ via a parametric approach, thereby reconciling \autoref{def:Sobolev-spaces-manifold} with that of Dziuk and Elliott \cite[Definition 2.11]{DziukElliott2013} for $\G$ of class $C^{1,1}$.
As an added benefit, we also infer that $\WW^{2,p}(\G)$-solutions of general scalar elliptic problems are also strong solutions even when $\G$ is only $C^{1,1}$;
see \autoref{subsec:main-results}.

\begin{proposition}[integration by parts on $C^{1,1}$-manifolds]
If $\G$ is of class $C^{1,1}$, then integration-by-parts formula \eqref{eq:int-by-parts} holds true.
\begin{proof}
For the rest of the proof we embrace the Einstein summation convention.
We restrict ourselves to proving that for each $u \in C^{1,1}(\G)$ and $\bvarphi \in \bC^{0,1}_t(\G)$ it holds that
\begin{equation}\label{eq:int-by-parts-tangential}
\int_\G u \, \divG \bvarphi = - \int_\G \nablaG u \cdot \bvarphi,
\end{equation}
from which \eqref{eq:int-by-parts} easily follows from $\bvarphi = \bP\bvarphi + (\bvarphi \cdot \bnu)\bnu$ for $\bvarphi \in \bC^1(\G)$ and the product rule for covariant derivatives
\begin{equation*}
\divG( (\bvarphi \cdot \bnu)\bnu )
= (\bvarphi \cdot \bnu) \divG(\bnu ) + \bnu \cdot \nablaG(\bvarphi \cdot \bnu)
= (\bvarphi \cdot \bnu) \tr(\bB),
\end{equation*}
and a density argument for $u \in \WW^{1,1}(\G)$.

Let us first consider an arbitrary $C^{1,1}$ chart $(\cV,\cU,\wt\bchi)$ of $\G$ and $u \in C^{1,1}(\G)$ compactly supported on $\cU \cap \G$.
Then, for each $\bvarphi \in \bC^{0,1}_t(\G)$, recalling identities \eqref{change-of-variable}, \eqref{surface-gradient-param} and \eqref{surface-divergence-param}, and integrating by parts we obtain in first instance that
\begin{multline*}
\int_\G u \, \divG \bvarphi
= \int_\cV \wt a_{\bchi} \wt u \, \wt g^{ij} \partial_i \wt\bchi \cdot \partial_j \wt\bvarphi
= - \int_\cV \partial_j (\wt a_{\bchi} \wt u \, \wt g^{ij} \partial_i \wt\bchi) \cdot  \wt\bvarphi
= - \int_\cV \wt a_{\bchi} (\partial_j \wt u) \wt g^{ij} \partial_i \wt\bchi \cdot  \wt\bvarphi - \int_\cV \wt u \partial_j (\wt a_{\bchi} \, \wt g^{ij} \partial_i \wt\bchi) \cdot  \wt\bvarphi\\
= - \int_\cV \wt a_{\bchi} (\nabla \wt\bchi \wt\bg^{-1} \nabla \wt u) \cdot \wt\bvarphi - \int_\cV \wt a_{\bchi} \wt u \, \wt{\DeltaG\bchi} \cdot  \wt\bvarphi
= - \int_\G \nablaG u \cdot \bvarphi - \int_\cV \wt a_{\bchi} \wt u \, \wt{\DeltaG\bchi} \cdot  \wt\bvarphi,
\end{multline*}
where $\DeltaG \bchi := (\DeltaG \chi_k)_{k=1}^{d+1}$.
Since the columns of $\nabla \bchi$ form a basis of the tangent plane to $\G$, it is enough to show that $\wt{\DeltaG \bchi} \cdot \partial_p \wt\bchi = 0$ a.e.~for each $p=1,\dotsc,d+1$ in order to deduce \eqref{eq:int-by-parts-tangential}.

In what follows, solely to avoid cluttered expressions, we omit the use of $\sim$ to express that computations are carried out in parametric domain.
Recalling that $a := a_{\bchi} = \sqrt{\det \bg}$ and utilizing Jacobi's formula it is easy to see that $\partial_j  a =  \frac{1}{2}  a g^{kl} \partial_j g_{kl}$ a.e.~and so ${\DeltaG \bchi}
= \frac{1}{a} \partial_j  a  g^{ij} \partial_i \bchi + \partial_j ( g^{ij} \partial_i \bchi)
= \frac{1}{2} g^{kl} \partial_j g_{kl} g^{ij} \partial_i \bchi + \partial_j g^{ij} \partial_i \bchi + g^{ij} \partial_j \partial_i \bchi$.
Hence, also using the facts that $g^{ij} g_{ip} = \delta_{jp}$ and $g_{ip} = \partial_i \bchi \cdot \partial_p \bchi$ and, by the product rule, $\partial_j g^{ij} g_{ip} = - g^{ij} \partial_j g_{ip}$ and $\partial_j g_{ip} =  \partial_j \partial_i \bchi \cdot \partial_p \bchi + \partial_i \bchi \cdot \partial_j \partial_p \bchi$, we deduce that a.e.~
\begin{multline*}
\DeltaG \bchi \cdot \partial_p \bchi
= \frac{1}{2} g^{kl} \partial_j g_{kl} g^{ij} g_{ip} + \partial_j g^{ij} g_{ip} + g^{ij} \partial_j \partial_i \bchi \cdot \partial_p \bchi
= \frac{1}{2} g^{kl} \partial_p g_{kl} + \partial_j g^{ij} g_{ip} + g^{ij} \partial_j \partial_i \bchi \cdot \partial_p \bchi\\
= \frac{1}{2} g^{ij} \partial_p g_{ij} - g^{ij} \partial_j g_{ip} + g^{ij} \partial_j g_{ip} - g^{ij} \partial_i \bchi \cdot \partial_j \partial_p \bchi
= \frac{1}{2} g^{ij}  \left( \partial_p \partial_i \bchi \cdot \partial_j \bchi - \partial_p \partial_j \bchi \cdot \partial_i \bchi \right)
= 0,
\end{multline*}
where the last equality is due to the symmetry of $g^{ij}$ and the skew-symmetry in $(i,j)$ of the term in parentheses.
This finishes the proof of \eqref{eq:int-by-parts-tangential} for $u$ compactly supported on $\cU \cap \G$.
For general $u \in C^{1,1}(\G)$, we proceed as follows:
let $\{\theta_i\}_{i=1}^N$ be a $C^{1,1}$ partition of unity associated with the covering $\{\cW_i\}_{i=1}^N$ of $\G$ (cf.~\autoref{sec:diff-geo}) and define $\mho_i := \theta_i u \in C^{1,1}(\G)$, which is supported on $\overline{\cW_i \cap \G} \subseteq \cU_i \cap \G$.
We then obtain
\begin{equation*}
\int_\G u \, \divG \bvarphi
= \sum_{i=1}^N \int_\G \mho_i \, \divG \bvarphi
= - \sum_{i=1}^N \int_\G \nablaG \mho_i \cdot \bvarphi
= - \int_\G \nablaG \left(\sum_{i=1}^N\mho_i\right) \cdot \bvarphi
= - \int_\G \nablaG u \cdot \bvarphi.
\end{equation*}
This finishes the proof.
\end{proof}
\end{proposition}
The following density result, which is a slight improvement over \cite[Lemma~3.7]{Miura-PartI}), will prove useful when dealing with vector-valued hypersurface PDEs in the second part of this project \cite{BenavidesNochettoShakipov2025-b}.
\begin{lemma}[density]\label{lem:density-tangential}
Let $p \in [1,\infty)$, $m \in \NN$ and an integer $0 \leq k \leq m-1$.
If $\G$ is of class $C^m$, then the space of everywhere tangential $C^{m-1}$-vector fields $\bC^{m-1}_t(\G)$ is dense in $\bW^{k,p}_t(\G)$.
Similarly, if $m \geq 2$ and $\G$ is of class $C^{m-1,1}$, then the space of a.e.~tangential $C^{m-2,1}$-vector fields $\bC^{m-2,1}_t(\G)$ is dense in $\bW^{k,p}_t(\G)$.
\begin{proof}
Assume $\G$ is of class $C^m$ (resp.~$C^{m-1,1}$ with $m \geq 2$) and let $\bu \in \bW^{k,p}_t(\G)$.
Then, by \autoref{thm:density-C^m} (density), there is a sequence $\{\bvarphi_n\}_{n \in \NN} \subseteq \bC^m(\G)$ (resp.~$\subseteq \bC^{m-1,1}(\G)$) such that $\lim_{n\to\infty} \|\bvarphi_n - \bu\|_{k,p;\G} = 0$.
Then, since $\bP = \bI - \bnu \otimes \bnu$ is of class $C^{m-1}$ (resp.~$C^{m-2,1}$), so is the tangential field $\bP \bvarphi_n$.
Consequently, noticing that $\bP\bu = \bu$ a.e.~on $\G$, we have by \autoref{lem:product-by-smooth} (product rule) that
\begin{equation*}
\|\bP\bvarphi_n - \bu\|_{k,p;\G}
= \|\bP(\bvarphi_n - \bu)\|_{k,p;\G}
\leq C \|\bvarphi_n - \bu\|_{k,p;\G}
\xrightarrow[]{n \to \infty} 0.
\end{equation*}
This finishes the proof.
\end{proof}
\end{lemma}

The following result will be employed in the proof of \autoref{thm:div-form-wp} (well-posedness of the operator $-\divG(\bsfA \nablaG \cdot)$ in $\WW^{1,p}_\#(\G)$) when the problem is written in parametric domain;
see \autoref{subsec:lp-reg-divform} for the proof.
\begin{remark}[tangential components]\label{rem:tangential-expansion}
If $p \in (1,\infty)$ and $\G$ is of class $C^{m,1}$ for some integer $m \geq 0$, then \autoref{lem:product-by-smooth} (product rule in $W^{m.p}(\G)$) implies: for each $\bv \in \bW^{m,p}_t(\G) := \WW^{m,p}_t(\G,\RR^{d+1})$ supported on $\cW \cap \G$, there exists a unique $\wh\bv = (\wh v_i)_{i=1}^d \in \bW^{m,p}(\cV) := \WW^{m,p}(\cV,\RR^d)$ supported on $\cX := \wt\bchi^{-1}(\cW \cap \G) \subseteq \cV$ such that
\begin{equation}\label{tangent-expansion}
\wt\bv = (\nabla\wt\bchi) \wh\bv \quad \text{on $\cV$},
\end{equation}
where $\wt\bv = \bv \circ \wt\bchi \in \WW^{m,p}(\cV,\RR^{d+1})$;
equivalently, $\wh\bv = \bg^{-1}(\nabla^T \wt\bchi) \wt\bv$ on $\cV$.
Moreover,
\begin{equation}\label{equiv-norms-manifold-param-tangentpart}
\|\wh\bv\|_{m,p;\cX} \cong \|\wt\bv\|_{\WW^{m,p}( \cX ,\RR^{d+1})} \cong \|\bv\|_{m,p;\G}.
\end{equation}
where the implicit equivalence constants depend only on $\G$ and $p$.

\begin{proof}
Identity \eqref{tangent-expansion} follows by noticing that the columns of $\nabla\wt\bchi$ form a basis of the tangent plane to $\G$.
In turn, \eqref{equiv-norms-manifold-param-tangentpart} is a consequence of \autoref{rem:equivalence-parametric-surface} (norm equivalence) and the Leibniz rule for Sobolev spaces in flat domains.
\end{proof}
\end{remark}

\subsection{Sobolev embeddings on \texorpdfstring{$\G$}{Γ}}
We now state various versions of Sobolev embedding results on $\G$.
For the rest of this section $p \in [1,\infty)$ and $m \in \NN$ are arbitrary unless stated otherwise, and whenever we mention the space $\WW^{m,p}(\G)$ we implicitly assume $\G$ is of class $C^{m-1,1}$.

Since $\G$ is compact, it follows that the Sobolev embeddings of the first type (Gagliardo--Nirenberg--Sobolev) on $\G$ behave just like their counterparts for a bounded open subset of $\RR^d$ with $C^1$ boundary.
\begin{proposition}[Gagliardo–Nirenberg–Sobolev inequality {\cite[\textsection 4.3]{HebeyRobert2008}}]\label{prop:GNG-embeddings}
Let $1 \leq p < q < \infty$, and any integers $0 \leq m < k$ be such that
\begin{equation*}
k - \frac{d}{p} \ge m - \frac{d}{q}.
\end{equation*}
If $\G$ is of class $C^{k-1,1}$, then $\WW^{k,p}(\G)$ is continuously embedded in $\WW^{m,q}(\G)$.

As a consequence, if $kp < d$ and $1 \leq q \leq \frac{dp}{d-kp}$, then $\WW^{k,p}(\G)$ is continuously embedded in $\LL^{q}(\G)$.
\end{proposition}
The Rellich--Kondrachov compactness theorem also extends from flat bounded domains to compact $\G$.
We refer to \cite[p.~389]{HebeyRobert2008} for the definition of the H\"older space $C^{0,\alpha}(\G)$.

\begin{proposition}[compact Sobolev embedding {\cite[Theorem 8.1]{HebeyRobert2008}}]\label{prop:Rellich-Kondrachov}
Let $1 \leq p < \infty$ and integers $0 \leq m < k$ be such that $kp < d$. If $\G$ is of class $C^{k-1,1}$, then $\WW^{k,p}(\G)$ is compactly embedded in $\WW^{m,q}(\G)$ for each $1 \leq q < \frac{dp}{d-(k-m)p}$.
In particular, $\WW^{1,p}(\G)$ is compactly embedded in $\LL^q(\G)$ for $1 \leq p < d$ and $1 \leq q < \frac{dp}{d-p}$ for $\G$ of class $C^{0,1}$.
\end{proposition}

\begin{proposition}[compact embedding in H\"older spaces {\cite[Theorem 8.2]{HebeyRobert2008}}]\label{prop:Holder-compactly-embedded}
Let $d < p < \infty$.
If $\G$ is of class $C^{0,1}$, then $\WW^{1,p}(\G)$ is compactly embedded in the H\"older space $C^{0,\alpha}(\G)$ for all $\alpha \in (0,1-\frac{d}{p})$.
\end{proposition}

\begin{remark}[compact Sobolev embedding]\label{rem:W1p-compactly-embedded-in-L^p}
Since $\frac{dp}{d-p} \to \infty$ as $p \to d$, we have from \autoref{prop:Rellich-Kondrachov} and \autoref{prop:Holder-compactly-embedded} that $\WW^{1,p}(\G)$ is compactly embedded in $\LL^p(\G)$ for all $1 \leq p < \infty$.
\end{remark}

\section{The scalar elliptic theory} \label{sec:scal-ell}
%
In this section we study the well-posedness and $\LL^p$-based Sobolev regularity for problem \eqref{eq:gsep}.
We first focus on problems in divergence form in \autoref{subsec:lp-reg-divform}.
In \autoref{subsec:lp-reg-GSES} we combine the results from \autoref{subsec:lp-reg-divform} with the Fredholm Alternative to derive the corresponding results for general scalar elliptic problems \eqref{eq:gsep}.

\subsection{\texorpdfstring{$\LL^p$}{Lᵖ}-based regularity for problems in divergence form} \label{subsec:lp-reg-divform}
%
Recall the notation for Lebesgue and Sobolev spaces $\LL^p_\#(\G)$ and $\WW^{m,p}_\#(\G)$ introduced in \autoref{sec:Sobolev-spaces-manifolds}.
The $\LL^2$-based weak formulation for problem $-\divG(\bsfA\nablaG \bu) = f$ reads as follows:
given $f \in (\HH^1_\#(\G))'$ and $\bsfA \in \bbL^\infty(\G)$ that satisfies the ellipticity condition \eqref{A-ellipticity-Gamma}, find $u \in \HH^1_\#(\G)$ such that
\begin{equation}\label{eq:H1-AdivnablaG}
\int_\G \bsfA \nablaG u \cdot \nablaG v = \langle f , v \rangle, \qquad \forall v \in \HH^1_\#(\G).
\end{equation}
The well-posedness of \eqref{eq:H1-AdivnablaG} for $\G$ of class $C^{0,1}$ follows from a combination of \eqref{A-ellipticity-Gamma} and Poincaré--Friedrichs inequality \cite[Lemma 2]{BonitoDemlowNochetto2020}, and the Lax--Milgram theorem.
In particular, there exists a constant $C>0$ depending only on $\G$ and $\bsfA$ such that
\begin{equation}\label{LaplaceBeltrami-apriori}
\|u\|_{1,\G} \leq C \|f\|_{(\HH^1_\#(\G))'}.
\end{equation}
We now embark on the task of establishing the well-posedness of the $\LL^p$-based analogue \eqref{eq:div-form-W1p-weakform} of problem \eqref{eq:H1-AdivnablaG}.
Our arguments are based on the ``localization" approach of \cite[Lemma 3]{BonitoDemlowNochetto2020} and a duality argument, and build upon the well-posedness in $\HH^1_\#(\G)$ of \eqref{eq:gsep}. 

One difficulty with writing the weak formulation \eqref{eq:div-form-W1p-weakform} in the parametric domain is the fact that $f \in (\WW^{1,p^*}_\#(\G))'$ does not necessarily make sense pointwise, so it is not immediately clear how to write it in the parametric domain as is. This is where the divergence form of the elliptic PDE becomes crucial---instead of mapping $f$ to the parametric domain, we shall construct $\bu_f \in \bL^p_t$ such that $\divG \bu_f = f$ as objects in $(\WW^{1,p^*}_\#)'$, and lift $\bu_f \in \bL^p_t$ instead. 
We thus first need to show that $\divG : \bL^p_t \to (\WW^{1,p^*}_\#)'$, whose definition is inspired by integration-by-parts formula \eqref{eq:int-by-parts-tangential}, is surjective.
This is the content of the following result.

\begin{proposition}[surjectivity of the divergence operator]\label{prop:surj-div}
Let $p \in (1,\infty)$.
If $\G$ is of class $C^{0,1}$, then the linear and continuous operator $\divG : \bL^p_t(\G) \rightarrow (\WW^{1,p^*}_\#(\G))'$ defined by
\begin{equation*}
\langle \divG \bu, v \rangle := -\int_\G \bu \cdot \nablaG v, \qquad \forall \bu \in \bL^p_t(\G), \ \forall v \in \WW^{1,p^*}_\#(\G),
\end{equation*}
is surjective.
\begin{proof}
The duality identification $(\bL^p_t(\G))' \equiv \bL^{p^*}_t(\G)$ and Poincaré inequality in $\WW^{1,p^*}_\#(\G)$ yield
\begin{equation*}
\sup_{\substack{\bu \in \bL^p_t(\G)\\ \bu \neq \mathbf{0}}} \frac{\langle \divG \bu, v \rangle}{\|\bu\|_{0,p;\G}}
= \sup_{\substack{\bu \in \bL^p_t(\G)\\ \bu \neq \mathbf{0}}} \frac{\int_\G \bu \cdot \nablaG v}{\|\bu\|_{0,p;\G}}
\cong \|\nablaG v\|_{0,p^*;\G}
\cong \|v\|_{1,p^*;\G},
\qquad \forall v \in \WW^{1,p^*}_\#(\G).
\end{equation*}
Hence, the desired assertion follows from item i) of \autoref{thm:inj-surj-rel} (characterization of injectivity and surjectivity).
\end{proof}
  
\end{proposition}

We are now in position to prove \autoref{thm:div-form-wp} (invertibility of the operator $-\divG(\bsfA \nablaG \cdot)$ in $\WW^{1,p}_\#(\G)$).

\begin{proof-thm}{{thm:div-form-wp}} 
We proceed in three steps, depending on the value of $p$ and the regularity of $\G$.

\medskip\noindent
\textbf{Step 1.}
We first argue for $p > 2$ and $\G$ of class $C^1$.
Since $f \in (\WW^{1,p^*}_\#(\G))' \subseteq (\HH^1_\#(\G))'$, the well-posedness of \eqref{eq:H1-AdivnablaG} yields the existence of a unique $u \in \HH^1_\#(\G)$ such that
\begin{equation}\label{H1-LB-rhs-better}
\int_\G \bsfA \nablaG u \cdot \nablaG v = \langle f, v \rangle, \qquad \forall v \in \HH^1_\#(\G).
\end{equation}
We aim to prove that $u$ is more regular, namely in $\WW^{1,p}_\#(\G)$, from which \eqref{eq:div-form-W1p-weakform} then easily follows from \eqref{H1-LB-rhs-better} and the density of $\HH^1_\#(\G)$ in $\WW^{1,p^*}_\#(\G)$ for the test space (cf.~\autoref{thm:density-C^m}).

We first notice that by \autoref{prop:surj-div} (surjectivity of the divergence operator) and \autoref{rem:inv-surj} we know there exists $\bu_f \in \bL^p_t(\G)$ such that
\begin{equation}\label{f-form-uf}
\langle f, v \rangle = - \int_\G \bu_f \cdot \nablaG v
\end{equation}
for each $v \in \WW^{1,p^*}_\#(\G)$, and
\begin{equation}\label{inv-div-bound}
\|\bu_f\|_{0,p;\G} \lesssim \|f\|_{(\WW^{1,p^*}_\#(\G))'}.  
\end{equation}
Moreover, identity \eqref{f-form-uf} lets us naturally extend $f$ from $\WW^{1,p^*}_\#(\G)$ to $\WW^{1,p^*}(\G)$.

We proceed as in the proof of \cite[Lemma 3]{BonitoDemlowNochetto2020} via a ``localization'' approach.
Let $\{\theta_i\}_{i=1}^N$ be a $C^1$ partition of unity associated with the covering $\{\cW_i\}_{i=1}^N$ of $\G$ \cite[Ch.~4, 11.~Theorem]{Shahshahani2016}.
We first argue with one chart, and hence we drop the index $i$.
It is easy to see that $\mho := \theta u \in \HH^1(\G)$ is compactly supported on $\cW \cap \G$ and satisfies
\begin{equation}\label{localized:H1-LB-rhs-better}
\int_\G \bsfA \nablaG \mho \cdot \nablaG v = \int_\G g v + \int_\G \bh \cdot \nablaG v, \qquad \forall v \in \HH^1(\G),
\end{equation}
where $g := - \bsfA\nablaG \theta \cdot \nablaG u - \nablaG \theta \cdot \bu_f$ and $\bh := u \bsfA \nablaG \theta - \theta \bu_f$ are also compactly supported on $\cW \cap \G$.
Recalling formulas \eqref{surface-gradient-param} and \eqref{Laplace--Beltrami-param} and \autoref{rem:tangential-expansion}, we express \eqref{localized:H1-LB-rhs-better} as
\begin{equation}\label{localized:H1-LB-rhs-better_parametric}
\int_\cV \wt a_{\bchi} \wt\bM \nabla \wt\mho \cdot \nabla \wt v
= \int_\cV \wt a_{\bchi} \wt g \, \wt v + \int_\cV \wt a_{\bchi} \hat\bh \cdot \nabla \wt v, \qquad \forall \wt v \in \HH^1(\cV),
\end{equation}
where $\wt\bM := \wt\bg^{-1} \nabla^T \wt\bchi \wt\bsfA \nabla \wt\bchi \, \wt\bg^{-1}$ and $\hat\bh$ denotes the ``tangential components'' of $\bh$ in the sense of \autoref{rem:tangential-expansion}.
We now proceed via a bootstrapping argument.
Note that $\wt\bg$ and $\wt a_{\bchi}$ are of class $C^0$ because $\G$ is of class $C^1$.
Then, if $u \in \WW^{1,\frac{2d}{d-2n}}(\G)$ for some integer $0 \leq n < d/2$, Sobolev embeddings (cf.~\cite[pp.~284-285]{Evans2010})
\begin{equation*}
\begin{cases}
\WW^{1,\frac{2d}{d-2n}}(\cV) \hookrightarrow \LL^q(\cV), & \qquad \text{for all $q \in (1,\infty)$, if $2(n+1) \geq d$},\\
\WW^{1,\frac{2d}{d-2n}}(\cV) \hookrightarrow \LL^{\frac{2d}{d-2(n+1)}}(\cV), & \qquad \text{if $2(n+1) < d$},
\end{cases}
\end{equation*}
imply that $\hat\bh \in \bL^{q_{n+1}}(\cV)$, where
\begin{equation*}
q_{n+1} := \begin{cases}
  p, & \text{if $2(n+1) \geq d$},\\
  \min\{p,\frac{2d}{d-2(n+1)}\}, & \text{if $2(n+1) < d$},
\end{cases}
\end{equation*}
with the estimate
\begin{equation}\label{bootstrapping-estimate-1}
\|\hat\bh\|_{0,q_{n+1};\cV}
\stackrel{\eqref{eq:norm-equivalence}}{\lesssim} \|\wt u\|_{1,\frac{2d}{d-2n};\cV} + \|\hat\bu_f\|_{0,p;\cV}.
\end{equation}
Moreover, recalling that $g = - \bsfA \nablaG \theta \cdot \nablaG u - \nablaG \theta \cdot \bu_f$, we also have $\wt g \in \LL^{q_n}(\cV)$ where $q_n = \min\{p,\frac{2d}{d-2n}\}$ (because $0 \leq n < d/2$) with a similar estimate to \eqref{bootstrapping-estimate-1}.
Consequently, from \autoref{prop:gen-el-sys} we deduce that $\wt\mho \in \WW^{1,s}(\cV)$ and
\begin{equation}\label{higherintegrability-a-priori-param-1}
\|\wt\mho\|_{1,s;\cV} \lesssim \|\hat\bh\|_{0,q_{n+1};\cV} + \|\wt g\|_{0,q_n;\cV}
\stackrel{\eqref{eq:norm-equivalence},\eqref{bootstrapping-estimate-1}}{\lesssim} \|\wt u\|_{1,\frac{2d}{d-2n};\cV} + \|\hat\bu_f\|_{0,p;\cV},
\end{equation}
where
\begin{align*}
s
& := 
\begin{cases}
\min\{q_{n+1}, \frac{d q_n}{d-q_n}\}, & \text{if $q_n < d$},\\
q_{n+1}, & \text{if $q_n \geq d$}.
\end{cases}
= q_{n+1}.
\end{align*}
In order to see this last equality, notice that if $\frac{2d}{d-2n}<d$ and $p<d$ (the case $p \geq d$ is easier), then because of the monotonicity of the map $x \mapsto \frac{dx}{d-x}$ for $x \in [1,d)$, it follows that
\begin{equation*}
\frac{dq_n}{d-q_n}
= \min\left\{\frac{dp}{d-p}, \frac{d \frac{2d}{d-2n}}{d-\frac{2d}{d-2n}}\right\}
= \min\left\{\frac{dp}{d-p}, \frac{2d}{d-2(n+1)}\right\}
\geq \min\left\{p,\frac{2d}{d-2(n+1)}\right\}
= q_{n+1}.
\end{equation*}
Combining estimate \eqref{higherintegrability-a-priori-param-1} for each of the charts we deduce that $u \in \WW^{1,q_{n+1}}(\G)$ and
\begin{multline*}
\|u\|_{1,q_{n+1};\G}
\leq \sum_{i=1}^N \|\mho_i\|_{1,q_{n+1};\G}
\stackrel{\eqref{eq:norm-equivalence}}{\lesssim} \sum_{i=1}^N \|\wt\mho_i\|_{1,q_{n+1};\cV_i}\\
\stackrel{\eqref{higherintegrability-a-priori-param-1}}{\lesssim} \sum_{i=1}^N \left(\|\wt u_i\|_{1,\frac{2d}{d-2n};\cV_i} + \|\hat\bu_f\|_{0,p;\cV_i}\right)
\stackrel{\eqref{equiv-norms-manifold-param-tangentpart},\eqref{scalar-Sobolev-norm}}{\cong} \|u\|_{1,\frac{2d}{d-2n};\G} + \|\bu_f\|_{0,p;\G}
\stackrel{\eqref{inv-div-bound}}{\lesssim} \|u\|_{1,\frac{2d}{d-2n};\G} + \|f\|_{(\WW^{1,p^*}_\#(\G))'}.
\end{multline*}
Recall that $u \in \WW^{1,2}(\G) = \HH^1(\G)$ satisfies \eqref{eq:div-form-W1p-weakform}.
Hence, by iterating over $n$ starting from $n=0$, and using that $\frac{2d}{d-x} \to \infty$ as $x\to d$, we eventually find (the first) $N$ such that $q_{N+1} = p$ and so
\begin{multline*}
\|u\|_{1,q_{N+1};\G}
\lesssim \|u\|_{1,\frac{2d}{d-2N};\G} + \|f\|_{(\WW^{1,p^*}_\#(\G))'}
= \|u\|_{1,q_N;\G} + \|f\|_{(\WW^{1,p^*}_\#(\G))'}\\
\lesssim \ldots
\lesssim \|u\|_{1,\G} + \|f\|_{(\WW^{1,p^*}_\#(\G))'}
\stackrel{\eqref{LaplaceBeltrami-apriori}}{\lesssim} \|f\|_{(\WW^{1,p^*}_\#(\G))'}.
\end{multline*}
This finishes the proof for $p>2$ and $\G$ of class $C^1$.

\medskip\noindent
\textbf{Step 2.}
If $p>2$, $\G$ is only of class $C^{0,1}$ and if $\bsfA$ lies only in $\bbL^\infty(\G)$, we can proceed as before with a $C^{0,1}$ partition of unity to arrive at the uniformly elliptic equation \eqref{localized:H1-LB-rhs-better_parametric} with coefficient matrix $\wt a_{\bchi} \wt\bM \in \bbL^\infty(\cV)$.
Then, because of \autoref{app:rem1} (Meyers), the bootstrapping argument still holds as long as $p$ lies in $[2,2+\eta)$ with $\eta$ depending only on $\G$ and $\bsfA$ (via the coefficients $\wt a_{\bchi} \wt\bM \in \bbL^\infty(\cV)$).

\medskip\noindent
\textbf{Step 3.}
For $p \in (1,2)$ we argue by duality;
roughly speaking, we will apply the case $p>2$ to the problem resulting from replacing $\bsfA$ by $\bsfA^T$.
Indeed, we just showed that $T : \WW^{1,p^*}_\#(\G) \rightarrow (\WW^{1,p}_\#(\G))'$ defined by $\langle T v, u \rangle = \int_\G \bsfA^T \nablaG v \cdot \nablaG u$ is a linear continuous bijection (because $p^* \in (2,\infty)$ and $\bsfA^T$ satisfies \eqref{A-ellipticity-Gamma}) if $\G$ is of class $C^1$ and $\bsfA \in \bbC(\G)$ or if $\G$ is of class $C^{0,1}$, $\bsfA \in \bbL^\infty(\G)$ and $p^*>2$ is sufficiently close to $2$.
Under these assumptions, we deduce from \autoref{thm:inj-surj-rel} (characterization of injectivity and surjectivity) that the adjoint $T^\mathrm{t} : \WW^{1,p}_\#(\G) \rightarrow (\WW^{1,p^*}_\#(\G))'$ defined by $\langle T^\mathrm{t} u, v \rangle := \langle T v, u \rangle = \int_\G \bsfA^T \nablaG v \cdot \nablaG u = \int_\G \bsfA \nablaG u \cdot \nablaG v$ is also a continuous bijection.
This finishes the proof for $p \in (1,2)$.
\end{proof-thm}

A straightforward combination of \autoref{thm:div-form-wp} with Sobolev embeddings (cf.~\autoref{sec:Sobolev-spaces-manifolds}) yields the following result.
In particular, we will utilize it in our companion paper \cite{BenavidesNochettoShakipov2025-b}.
\begin{corollary}\label{cor:LB-W1p-special-RHS}
Let $p \in (1,\infty)$ and $q \in (1,d)$, and define $s := \min\{p,\frac{dq}{d-q}\}$.
If $\G$ is of class $C^1$ and $\bsfA \in \bbC(\G)$ satisfy the ellipticity condition \eqref{A-ellipticity-Gamma},
then for each $\bf \in \bL^p_t(\G)$ and $f \in \LL^q_\#(\G)$, there is a unique $u \in \WW^{1,s}_\#(\G)$ such that
\begin{equation}\label{eq:LaplaceBeltrami-W1p-special-RHS}
\int_\G \bsfA \nablaG u \cdot \nablaG v = - \int_\G \bf \cdot \nablaG v + \int_\G f \, v, \qquad \forall v \in \WW^{1,s^*}_\#(\G).
\end{equation}
Moreover, there exists a positive constant $C$, depending only on $\G$, $\bsfA$, $p$ and $q$ such that
\begin{equation*}
\|u\|_{1,s;\G} \leq C \left( \|\bf\|_{0,p;\G} + \|f\|_{0,q;\G} \right).
\end{equation*}
If $\G$ and $\bsfA$ are, respectively, only of class $C^{0,1}$ and $\bbL^\infty(\G)$, there exists $\varepsilon>0$, depending only on $\G$ and $\bsfA$, such that the result still holds as long as $s \in (2-\varepsilon,2+\varepsilon)$.
\end{corollary}
We next establish the higher $\LL^p$-based regularity of solutions to problem \eqref{eq:div-form-W1p-weakform}, namely \autoref{thm:div-form-higreg} ($\WW^{m+2,p}$-regularity for the operator $-\divG(\bsfA \nablaG \cdot)$ on $\G$).
Similarly to \autoref{thm:div-form-wp}, the proof also follows by ``localization".

\begin{proof-thm}{{thm:div-form-higreg}}
We first argue with one chart, and hence we drop the index $i$.
Similarly to the proof of \autoref{thm:div-form-wp}, it is easy to see that the ``localized'' variable $\mho := \theta u \in \WW^{1,p}_\#(\G)$ is compactly supported on $\cW \cap \G$ and satisfies
\begin{equation}\label{localized-LaplaceBeltrami}
\int_\G \bsfA\nablaG \mho \cdot \nablaG v = \int_\G g v, \qquad \forall v \in \WW^{1,p^*}(\G),
\end{equation}
where
$$
g := \theta f - \bsfA\nablaG\theta \cdot \nablaG u - \divG (u \bsfA \nablaG \theta) = \theta f - (\bsfA + \bsfA^T)\nablaG u \cdot \nablaG \theta - u \divG (\bsfA \nablaG \theta)
$$
is also compactly supported on $\cW_i \cap \G$.
Fix $k=0,\dots,m$ and assume by induction that $u \in \WW^{k+1,p}(\G)$ and that
\begin{equation}\label{ind_hyp-lp-hireg}
\|u\|_{k+1,p;\G} \leq C \|f\|_{\max\{0,k-1\},p;\G};
\end{equation}
this is valid for $k=0$ in view of \eqref{eq:div-form-W1p-apriori} and $\|f\|_{(\WW^{1,p^*}_\#(\G))'} \leq C \|f\|_{0,p;\G}$.
Then, $g \in \WW^{k,p}(\G)$ and $\|g\|_{k, p;\cW \cap \G} \lesssim \|f\|_{k, p;\cW \cap \G} + \|u\|_{k+1,p;\cW \cap \G}$.
Recalling \eqref{change-of-variable} and \eqref{surface-gradient-param}, we express \eqref{localized-LaplaceBeltrami} in parametric domain as
\begin{equation}\label{loc_lb_hi_reg}
\int_{\cV} \wt a_{\bchi} \wt\bM \nabla \wt \mho \cdot \nabla \wt v = \int_{\cV} \wt a_{\bchi} \wt g \, \wt v, \qquad \forall \wt v \in \WW^{1,p^*}(\cV),
\end{equation}
where $\wt\bM := \wt\bg^{-1} \nabla^T \wt\bchi \wt\bsfA \nabla \wt\bchi \, \wt\bg^{-1}$.
Since $\wt a_{\bchi} \wt\bM$ is uniformly elliptic and of class $C^{m,1}$ (because of \eqref{A-ellipticity-Gamma} and the facts that $\bsfA \in \bbC^{m,1}(\G)$ and $\G$ is of class $C^{m+1,1}$) and $\wt g \in \WW^{k,p}(\cV)$ (because $g \in \WW^{k,p}(\G)$), we deduce from \autoref{app:prop-hi-reg} with $s=p$ that $\wt\mho \in \WW^{k+2,p}(\cV)$ and
\begin{equation*}
\|\wt\mho\|_{k+2,p;\cV}
\lesssim \|\wt g\|_{k,p;\cV}
\stackrel{\eqref{eq:norm-equivalence}}{\lesssim} \|g\|_{k,p;\cU \cap \G}
\lesssim \|f\|_{k,p;\cU \cap \G} + \|u\|_{k+1,p;\cU \cap \G}.
\end{equation*}
Combining this last estimate for each of the charts of $\G$ we deduce that $u \in \WW^{k+2,p}(\G)$ and
\begin{multline*}
\|u\|_{k+2,p;\G}
\leq \sum_{i=1}^N \|\mho_i\|_{k+2,p;\G}
\stackrel{\eqref{eq:norm-equivalence}}{\lesssim} \sum_{i=1}^N \|\wt\mho_i\|_{k+2,p;\cV_i}\\
\lesssim \sum_{i=1}^N (\|f\|_{k,p;\cU_i \cap \G} + \|u\|_{k+1,p; \cU_i \cap \G})
\lesssim \|f\|_{k,p;\G} + \|u\|_{k+1,p;\G}
\stackrel{\eqref{ind_hyp-lp-hireg}}{\lesssim} \|f\|_{k,p;\G}.
\end{multline*}
This concludes the induction step, and hence, finishes the proof of estimate \eqref{eq:div-form-higher-apriori}.

Finally, integration-by-parts formula \eqref{eq:int-by-parts} applied to the left-hand side of \eqref{eq:div-form-W1p-weakform} and the density of $\WW^{1,p^*}(\G)$ in $\LL^{p^*}(\G)$ yields \eqref{div-form-a.e.}.
\end{proof-thm}

It is sometimes useful to consider a weaker notion of solution than $u \in \WW^{1,p}_\#(\G)$ for a right-hand side datum $f$ that is not in $(\WW^{1,p^*}_\#(\G))'$, but rather in $(\WW^{2,p^*}_\#(\G))'$.
We call these solutions ``ultra-weak''.
Examples of such applications include, e.g., optimal control problems where the datum is a measure, and the analysis of tangential vector-valued problems on manifolds \cite{BenavidesNochettoShakipov2025-b}. 
\begin{lemma}[ultra-weak solutions of the $-\divG(\bsfA \nablaG \cdot)$ operator on $\G$]\label{lem:ultraweak-allp-LB}
Let $p \in (1,\infty)$,
$\G$ be of class $C^{1,1}$ and $\bsfA \in \bbC^{0,1}(\G)$ satisfy \eqref{A-ellipticity-Gamma}.
Then, for each $f \in (\WW^{2,p^*}_\#(\G))'$ there exists a unique $u \in \LL^p_\#(\G)$ such that
\begin{equation}\label{eq:dual-problem}
-\int_\G u \, \divG(\bsfA^T \nablaG v) = \langle f, v \rangle, \qquad \forall v \in \WW^{2,p^*}_\#(\G).
\end{equation}
Moreover, there exists a constant $C>0$ depending only on $\G$, $\bsfA$ and $p$ such that
\begin{equation}\label{eq:dual-problem-apriori}
\|u\|_{0,p;\G} \leq C \|f\|_{(\WW^{2,p^*}_\#(\G))'}.
\end{equation}
\begin{proof}
Since the operator $T := -\divG(\bsfA\nablaG \cdot)$ is a continuous bijection from $\WW^{2,p^*}_\#(\G)$ onto $\LL^{p^*}_\#(\G) \equiv (\LL^p_\#(\G))'$, according to~\autoref{thm:div-form-higreg},
it follows from \autoref{app:var-prob} that the adjoint $T^\mathrm{t} : \LL^p_\#(\G) \rightarrow (\WW^{2,p^*}_\#(\G))'$ defined by $\langle T^\mathrm{t} u , v \rangle := \langle Tv, u \rangle = - \int_\G u \, \divG(\bsfA^T \nablaG v) $ is also a continuous bijection.
This finishes the proof.
\end{proof}
\end{lemma}

\subsection{\texorpdfstring{$\LL^p$}{Lᵖ}-based regularity for general scalar elliptic problems} \label{subsec:lp-reg-GSES}
%
This section is devoted to establishing \autoref{thm:gsep-wp} (well-posedness) and \autoref{thm:gsep-reg} (higher regularity).
We begin by establishing the uniqueness of the variational formulation associated to problem \eqref{eq:gsep-weakform} for $p \in [2,\infty)$ provided \eqref{cond-coeffs-with-b} holds true.
Our arguments are inspired by the proof of the Maximum Principle for weak solutions of elliptic equations in flat domains \cite[Section 8.1]{GilbargTrudinger2001}.

\begin{proposition}[uniqueness for $p \geq 2$]\label{prop:uniqueness-GSES}
Let $\G$ be of class $C^{0,1}$, let $p \in [2,\infty)$, $(\bsfA, \bsfb, \bsfc, \sfd) \in \bbL^\infty(\G) \times \bL^\infty_t(\G) \times \bL^\infty_t(\G) \times \LL^\infty(\G)$ and let conditions \eqref{A-ellipticity-Gamma} and \eqref{cond-coeffs-with-b} hold.
If $u \in \WW^{1,p}(\G)$ satisfies
\begin{equation}\label{u-kernel-GSES}
\int_\G \left(\bsfA \nablaG u \cdot \nablaG v + u (\bsfb \cdot \nablaG v)  + (\bsfc \cdot \nablaG u) v + \sfd \, u v\right)  = 0, \qquad \forall v \in \WW^{1,p^*}(\G),
\end{equation}
then $u = 0$ a.e.~on $\G$.
\begin{proof}
It is enough to consider the case $p=2$.
Let $u \in \HH^1(\G)$ satisfy \eqref{u-kernel-GSES}.

Notice that for each $v \in \HH^1(\G)$ we have $w := uv \in \WW^{1,1}(\G)$ and $\nablaG w = v \nablaG u + u \nablaG v$ a.e.~on $\G$ (cf.~\autoref{def:diff-ops-SSS} and \cite[pp.~150-151]{GilbargTrudinger2001}).
Hence, for each $v \in \HH^1(\G)$ such that $w \geq 0$ a.e.~on $\G$, condition \eqref{cond-coeffs-with-b} implies
\begin{equation}\label{proto:ineq-uv-v-geq0}
\begin{aligned}
0
&= \int_\G \left(\bsfA \nablaG u \cdot \nablaG v + u (\bsfb \cdot \nablaG v)  + (\bsfc \cdot \nablaG u) v + \sfd \, u v\right)\\
&= \int_\G \left(\bsfA \nablaG u \cdot \nablaG v + u (\bsfb \cdot \nablaG v)  + (\bsfc \cdot \nablaG u) v - \bsfb \cdot \nablaG(uv)\right) + \int_\G (\sfd \, u v + \bsfb \cdot \nablaG(uv))\\
&\geq \int_\G \left(\bsfA \nablaG u \cdot \nablaG v  + v(\bsfc - \bsfb)\cdot \nablaG u \right) + \lambda \int_M uv,
\end{aligned}
\end{equation}
whence
\begin{equation}\label{ineq-uv-v-geq0}
\int_\G \bsfA \nablaG u \cdot \nablaG v + \lambda \int_M uv \leq C \int_\G |\nablaG u| |v|.
\end{equation}
Denote $M := \operatorname{esssup}_{x \in \G} u(x) \in (-\infty,\infty]$ and suppose by contradiction that $M > 0$.
Then, for each $0 < k < M$, define $v_k := (u-k)^+ := \max\{u-k,0\}$.
Recalling \autoref{def:Sobolev-spaces-manifold} and standard results for Sobolev spaces in flat domains \cite[Lemma 7.6]{GilbargTrudinger2001}, we have that $v_k \in \HH^1(\G)$ and
\begin{equation*}
\nablaG v_k := \begin{cases}
\nablaG u, & \text{if $u>k$},\\
0, & \text{if $u \leq k$}.
\end{cases}
\end{equation*}
Consequently, from \eqref{ineq-uv-v-geq0} and noticing that $u-k = v_k \leq u$ if $u \geq k$, we have
\begin{equation}\label{ineq-vk}
\int_\G \bsfA \nablaG v_k \cdot \nablaG v_k + \lambda \int_M v_k^2
\leq \int_\G \bsfA \nablaG u \cdot \nablaG v_k + \lambda \int_M u\,v_k
\stackrel{\eqref{ineq-uv-v-geq0}}{\leq} C \int_{\gamma_k} |\nablaG v_k| |v_k|,
\end{equation}
where $\gamma_k := \{x \in \G: (\nablaG v_k)(x) \neq 0\} \subseteq \{x \in \G: v_k(x) > 0\}$.
Then, it is easy to prove (using the compactness of $\HH^1(\G)$ into $\LL^2(\G)$, similarly to \cite[Lemma 2]{BonitoDemlowNochetto2020}) that $v \mapsto \left(\int_\G \bsfA \nablaG v \cdot \nablaG v + \lambda \int_M v^2\right)^{\frac{1}{2}}$ is an equivalent norm in $\HH^1(\G)$.
Therefore, we further obtain from \eqref{ineq-vk} that
\begin{equation*}
\|v_k\|_{1,\G}^2
\lesssim \int_{\gamma_k} |\nablaG v_k| |v_k|
\leq \|v_k\|_{1,\G} \|v_k\|_{0,\gamma_k}.
\end{equation*}
In turn, for any fixed $q \in (2,\frac{2d}{d-2}]$ if $d \geq 3$ or $q \in (2,\infty)$ if $d=2$, the continuous injection $\HH^1(\G) \hookrightarrow \LL^q(\G)$ (cf.~\autoref{sec:Sobolev-spaces-manifolds}) and H\"older's inequality further yield
\begin{equation*}
\|v_k\|_{0,q;\G} \lesssim \|v_k\|_{1,\G} \lesssim \|v_k\|_{0,\gamma_k} \leq |\gamma_k|^{\frac{q-2}{2q}} \|v_k\|_{0,q;\G},
\end{equation*}
and so, since $\|v_k\|_{0,q;\G} > 0$,
\begin{equation*}
|\gamma_k| \geq C, \qquad \forall k < M,
\end{equation*}
for a constant $C > 0$ independent of $k$.

We claim that $M < \infty$.
Otherwise, by Chebyshev's inequality
\begin{equation*}
C
\leq \liminf_{n \to \infty} |\gamma_n|
\leq \liminf_{n \to \infty} |\{v_n > 0\}|
\leq \liminf_{n \to \infty} |\{u > n\}|
\leq \liminf_{n \to \infty} \frac{1}{n} \|u\|_{0,1;\G} = 0,
\end{equation*}
which is a contradiction.
Since $|\{u = M: \nablaG u \neq 0\}| = 0$, according to \cite[Lemma 7.7]{GilbargTrudinger2001} and its relation to \autoref{def:Sobolev-spaces-manifold},
we obtain
\begin{equation*}
|\gamma_M|
= |\{\nablaG v_M \neq 0\}|
= |\{u > M: \nablaG u \neq 0\}|
= |\{u \geq M: \nablaG u \neq 0\}|\\
= \left|\bigcap_{n \in \NN} \gamma_{M - \frac{1}{n}}\right| 
= \lim_{n \to \infty} \big|\gamma_{M - \frac{1}{n}} \big|
\geq C,
\end{equation*}
by continuity of the Lebesgue measure on $\G$.
However, $|\gamma_M| \leq |\{u > M \}| = 0$, because $u \leq M$ a.e.~on $\G$,
which is a contradiction; hence $M \leq 0$ and $u \leq 0$ a.e.~on $\G$.
Since \eqref{u-kernel-GSES} is linear in $u$, applying this estimate to $-u$, instead of $u$, yields $-u\le0$ a.e. on $\G$, whence $u = 0$ a.e.~on $\G$ as asserted.
\end{proof}
\end{proposition}

We now proceed to prove \autoref{thm:gsep-wp} (well-posedness of general scalar elliptic equations). Even though we follow the general argument used in flat domains, the lack of boundary of $\G$ plays against us.
As a result, conditions \eqref{cond-coeffs-with-b} and \eqref{cond-coeffs-with-c} on coefficients $\bsfb, \bsfc, \sfd$ appear to be inevitable.
Moreover, in order to retain both of these conditions for any choice of $p \in (1, \infty)$, we use a delicate argument shown below.

\begin{proof-thm}{{thm:gsep-wp}}
Roughly speaking, the argument goes as follows. For $p=2$, we establish the well-posedness of \eqref{eq:gsep-weakform} assuming \eqref{cond-coeffs-with-b} using \autoref{prop:uniqueness-GSES} and the Fredholm alternative. We then recognize that the same argument still applies to the adjoint problem, wherein $(\bsfA, \bsfb, \bsfc, \sfd)$ are replaced by $(\bsfA^T, \bsfc, \bsfb, \sfd)$---this yields a second sufficient condition, namely, \eqref{cond-coeffs-with-c}. At this point, for $p=2$, we have the well-posedness of \eqref{eq:gsep-weakform} when at least one of the conditions \eqref{cond-coeffs-with-b}, \eqref{cond-coeffs-with-c} holds. We obtain the well-posedness for $p \in (2,\infty)$ under the same conditions by bootstrapping integrability $p$ based upon \autoref{thm:div-form-wp} (well-posedness of the operator $-\divG(\bsfA \nablaG \cdot)$ in $\WW^{1,p}_\#(\G)$); the analogous result for $p<2$ follows by duality. The well-posedness for $\G$ of class $C^{0,1}$ and $p \in (2-\varepsilon, 2+\varepsilon)$ follows from Meyers' argument.
For readability purposes, we divide the proof into three cases.

\medskip\noindent\textbf{Case 1:}
$\G$ is of class $C^{0,1}$ and $p=2$.
We further split the proof according to whether condition \eqref{cond-coeffs-with-b} or \eqref{cond-coeffs-with-c} hold.

\smallskip
\noindent\textbf{Subcase 1.1: Condition \eqref{cond-coeffs-with-b}.}
Recalling the algebraic manipulations performed in \eqref{proto:ineq-uv-v-geq0} and \eqref{ineq-uv-v-geq0} within the proof of \autoref{prop:uniqueness-GSES}, we observe that upon defining the linear and bounded operators $L, K: \HH^1(\G) \rightarrow (\HH^1(\G))'$ by
\begin{align*}
\langle Lu,v \rangle & := \int_\G (\bsfA \nablaG u \cdot \nablaG v) + \int_\G (\sfd \, u v + \bsfb \cdot \nablaG(uv)), \\
\langle Ku,v \rangle & := \int_\G v(\bsfc - \bsfb)\cdot \nablaG u,
\end{align*}
we can equivalently express \eqref{eq:gsep-weakform} as:
find $u \in \HH^1(\G)$ such that
\begin{equation}\label{eq:bij-plus-comp}
(L+K)u = f.
\end{equation}
Taking $v=u$, arguing as in the proof of \autoref{prop:uniqueness-GSES} and using condition \eqref{cond-coeffs-with-b} gives the Poincaré-type inequality $\langle Lu,u \rangle \ge \Lambda \int_\G |\nablaG u|^2 + \lambda \int_M |u|^2$.
This in turn implies that $L$ is a continuous bijection from $\HH^1(\G) \rightarrow (\HH^1(\G))'$.

On the other hand, regarding the operator $K$, H\"older's inequality yields
\begin{equation*}
|\langle Ku,v \rangle| \leq C \|\nablaG u\|_{0,\G} \|v\|_{0,\G}, \qquad \forall u,v \in \HH^1(\G),
\end{equation*}
and so
\begin{equation*}
\|K^\mathrm{t}(v)\|_{(\HH^1(\G))'} \leq C \|v\|_{0,\G}, \qquad \forall v \in \HH^1(\G),
\end{equation*}
where $K^\mathrm{t}: \HH^1(\G) \rightarrow (\HH^1(\G))'$ is the ``adjoint'' of $K$ (see \autoref{app:var-prob}).
Consequently, since the embedding $\HH^1(\G) \hookrightarrow \LL^2(\G)$ is compact (cf.~\autoref{rem:W1p-compactly-embedded-in-L^p}) we deduce that $K^\mathrm{t}$ is compact and, by \cite[Theorem 6.4]{Brezis2011}, so is $K$.

Since $L$ is bijective and $K$ is compact, it follows from the Fredholm's alternative for reflexive Banach spaces \cite[Theorem 6.6]{Brezis2011} (see also \cite[Teorema 6.9]{Gatica2024}) that $L+ K$ has closed range and $\dim \operatorname{ker}(L+K) = \dim \operatorname{ker}((L+K)^\mathrm{t}) < \infty$.
In order to conclude that $L+K$ is invertible, it suffices to assert that $\ker(L+K) = \{\mathbf{0}\}$, which is a consequence of \autoref{prop:uniqueness-GSES}.
Finally, the Bounded Inverse Theorem implies \eqref{eq:gsep-apriori} and finishes the proof of Subcase 1.1.

\medskip\noindent\textbf{Subcase 1.2: Condition \eqref{cond-coeffs-with-c}.}
%
We argue by duality.
Let us define the linear operator $T: \HH^1(\G) \rightarrow (\HH^1(\G))'$ by
\begin{equation*}
\langle Tv, u \rangle :=
\int_\G \bsfA^T \nablaG v \cdot \nablaG u + v (\bsfc \cdot \nablaG u)  + (\bsfb \cdot \nablaG v) u + \sfd \, v u.
\end{equation*}
This is similar to the operator in \eqref{u-kernel-GSES} upon replacing $\bsfA$, $\bsfb$ and $\bsfc$ respectively by $\bsfA^T$, $\bsfc$ and $\bsfb$. Since \eqref{cond-coeffs-with-b} becomes \eqref{cond-coeffs-with-c} in the present context,
applying Subcase 1.1 implies that $T$ is an isomorphism.
Consequently, from \autoref{thm:inj-surj-rel} (characterization of injectivity and surjectivity) it transpires that $T^\mathrm{t}: \HH^1(\G) \rightarrow (\HH^1(\G))'$ defined by
\begin{align*}
\langle T^\mathrm{t} u, v \rangle
:= \langle Tv, u \rangle
& = \int_\G \bsfA^T \nablaG v \cdot \nablaG u + v (\bsfc \cdot \nablaG u)  + (\bsfb \cdot \nablaG v) u + \sfd \, v u\\
& = \int_\G \bsfA \nablaG u \cdot \nablaG v + u(\bsfb \cdot \nablaG v) + (\bsfc \cdot \nablaG u)v   + \sfd \, u v
\end{align*}
is also an isomorphism and \eqref{eq:gsep-apriori} is valid.
This finishes the proof of Subcase 1.2.

\medskip\noindent\textbf{Case 2:}
Let $p \in (2,\infty)$ and either $\G$ be of class $C^1$ or $\G$ be of class $C^{0,1}$ and $p$ sufficiently close to $2$.
If either \eqref{cond-coeffs-with-b} or \eqref{cond-coeffs-with-c} is satisfied, then the unique solvability of problem \eqref{eq:gsep-weakform} with a priori bound \eqref{eq:gsep-apriori} follows from Case 1 ($p=2$) and a bootstrapping argument based upon \autoref{thm:div-form-wp} (invertibility of the operator $-\divG(\bsfA \nablaG \cdot)$ in $\WW^{1,p}_\#(\G)$).

\medskip\noindent\textbf{Case 3:}
Let $p \in (1,2)$ and either $\G$ be of class $C^1$ or $\G$ be of class $C^{0,1}$ and $p$ sufficiently close to $2$. If either \eqref{cond-coeffs-with-b} or \eqref{cond-coeffs-with-c} holds, then reinterpreting the linear operator $T$ defined in Subcase 1.2 as $T: \WW^{1,p^*}(\G) \rightarrow (\WW^{1,p}(\G))'$ it follows from Case 2 (because $p^* > 2$) that $T$ is an isomorphism, whence so is $T^\mathrm{t}: \WW^{1,p}(\G) \rightarrow (\WW^{1,p^*}(\G))'$ and \eqref{eq:gsep-apriori} is valid.
This concludes the proof of Case 3, and hence the overall proof.
\end{proof-thm}

We now establish the higher $\LL^p$-based Sobolev regularity enjoyed by problem \eqref{eq:gsep-weakform} (cf.~\autoref{thm:gsep-reg}).

\begin{proof-thm}{{thm:gsep-reg}}
Using the integration-by-parts formula \eqref{eq:int-by-parts} (and \autoref{lem:product-by-smooth} (product rule on $\WW^{m,p}(\G)$), it is easy to see that the solution $u \in \WW^{1,p}(\G)$ of \eqref{eq:gsep-weakform} also satisfies the divergence-form equation
\begin{equation*}
\int_\G \bsfA \nablaG u \cdot \nablaG v = \int_\G h \, v, \qquad \forall v \in \WW^{1,p^*}(\G),
\end{equation*}
where $h := f + u \divG \bsfb + \bsfb \cdot \nablaG u - \bsfc \cdot \nablaG u - \sfd u$.
Consequently, a straightforward bootstrapping argument based upon \autoref{thm:div-form-higreg} ($\WW^{m+2,p}$-regularity for the operator $-\divG(\bsfA \nablaG \cdot)$ on $\G$) yields $u \in \WW^{m+2,p}(\G)$ and estimate \eqref{eq:gsep-higher-apriori}.
On the other hand, utilizing the integration-by-parts formula \eqref{eq:int-by-parts} applied to \eqref{eq:gsep-weakform} and the density of $\WW^{1,p^*}(\G)$ in $\LL^{p^*}(\G)$, we obtain the strong form \eqref{gsep-a.e.} of the PDE.
This concludes the proof.
\end{proof-thm}

\subsection{Extensions of \autoref{thm:gsep-wp} and \autoref{thm:gsep-reg}}\label{subsec:extensions}

In this section we discuss \autoref{cor:gsep-wp-b=0} and \autoref{thm:gset-wp-div-free-conv}, which explore further the solvability of general scalar elliptic equations.
\autoref{cor:gsep-wp-b=0} concerns the classical {\it convection-diffusion-reaction} equation
\begin{equation*}
-\divG(\bsfA \nablaG u) + \bsfc \cdot \nablaG u + \sfd u = f,
\end{equation*}
for which we can deduce, as a consequence of \autoref{thm:gsep-wp} and \autoref{thm:gsep-reg}, the $\LL^p$-based well-posedness and higher regularity assuming that $\sfd \geq 0$ is strictly positive on a subset of positive measure on $\G$ and $\bsfc$ is only in $\bL^\infty_t(\G)$. The proof reduces to verifying assumption \eqref{cond-coeffs-with-b}.

On the other hand, \autoref{thm:gset-wp-div-free-conv} is concerned with the solvability of the {\it convection-diffusion} equation
\begin{equation*}
-\divG(\bsfA \nablaG u) + \bsfc \cdot \nablaG u = f,
\end{equation*}
provided the convective term $\bsfc$ is $\divG$-free.
This PDE can be viewed as a special instance of the general equation \eqref{eq:gsep} for $(\bsfb,\sfd) = (\mathbf{0},0)$, for which \eqref{cond-coeffs-with-b} fails.
Moreover, \eqref{cond-coeffs-with-c} is ruled out by simply considering $w$ to be constant.
Therefore, \autoref{thm:gsep-wp} does not apply in this case.

Recall the notation: $\bC^{-1,1}_t(\G) := \bL^\infty_t(\G)$ and $\CC^{-1,1}(\G) := \LL^\infty(\G)$.

\begin{corollary}[convection-diffusion-reaction] \label{cor:gsep-wp-b=0}
Let $p \in (1,\infty)$ and let $\G$ be of class $C^1$. Consider coefficients $(\bsfA, \bsfc, \sfd) \in \bbC(\G) \times \bL^\infty_t(\G) \times \LL^\infty(\G)$, with $\bsfA$ satisfying the ellipticity condition \eqref{A-ellipticity-Gamma}, and $\sfd \geq 0$ a.e.~such that $\sfd > 0$ on a subset of positive measure on $\G$.
Then, for every $f \in (\WW^{1,p^*}(\G))'$, there exists a unique $u \in \WW^{1,p}(\G)$ such that
\begin{equation}\label{eq:gsep-weakform-without-b}
\int_\G \bsfA \nablaG u \cdot \nablaG v + (\bsfc \cdot \nablaG u) v + \sfd \, u v  = \langle f, v \rangle, \qquad \forall v \in \WW^{1,p^*}(\G).
\end{equation}
Moreover, there exists a constant $C>0$, depending only on $\G$, $\bsfA$, $\bsfc$, $\sfd$ and $p$, such that
\begin{equation}\label{eq:gsep-apriori-without-b}
\|u\|_{1,p;\G} \leq C \|f\|_{(\WW^{1,p^*}(\G))'}.
\end{equation}
This result is still valid for $\G$ of class $C^{0,1}$ and $\bsfA \in \bbL^\infty(\G)$ provided $p \in (2-\varepsilon, 2+\varepsilon)$ for $\varepsilon > 0$ sufficiently small depending only on $\G$ and $\bsfA$.

If in addition to the assumptions above, $\G$ is of class $C^{m+1,1}$ for some nonnegative integer $m$, and $(\bsfA, \bsfc, \sfd) \in \bbC^{m,1}(\G) \times \bC_t^{m-1,1}(\G) \times C^{m-1,1}(\G)$, then, for each $f \in \WW^{m,p}(\G)$, the solution $u \in \WW^{1,p}(\G)$ of \eqref{eq:gsep-weakform-without-b} is in fact in $\WW^{m+2, p}(\G)$
and satisfies
\begin{equation}\label{gsep-a.e.-without-b}
- \divG (\bsfA \nablaG u) + \bsfc \cdot \nablaG u + \sfd u = f, \qquad \text{a.e.~on $\G$}.
\end{equation}
Moreover, there is a constant $C >0$, depending only on $\G$, $\bsfA$, $\bsfc$, $\sfd$ and $p$, such that
\begin{equation}\label{eq:gsep-higher-apriori-without-b}
\|u\|_{m+2,p;\G} \leq C \|f\|_{m,p;\G}.
\end{equation}
\begin{proof}
Since $0 < |\{\sfd > 0\}| = \left| \bigcup_{n \in \NN} \{\sfd \geq \frac{1}{n}\} \right| = \lim_{n \to \infty} |\{\sfd \geq \frac{1}{n}\}|$, there exists $N \in \NN$ such that $M := \{x \in \G: \sfd(x) \geq \frac{1}{N}\}$ has positive measure.
Therefore, for each $w \in \WW^{1,1}(\G)$ such that $w \geq 0$ a.e.~on $\G$ it holds that
\begin{equation*}
\int_\G \sfd\,w \geq \int_M \sfd\,w \geq \frac{1}{N} \int_M w,
\end{equation*}
which shows that problem \eqref{eq:gsep-weakform-without-b} satisfies condition \eqref{cond-coeffs-with-b} with $\bb = \mathbf{0}$.
The proof then follows from straightforward applications of \autoref{thm:gsep-wp} and \autoref{thm:gsep-reg}.
\end{proof}
\end{corollary}

\begin{thm}[convection-diffusion] \label{thm:gset-wp-div-free-conv}
Let $p \in (1,\infty)$ and let $\G$ be of class $C^1$. Consider coefficients $(\bsfA, \bsfc) \in \bbC(\G) \times \bL^\infty_t(\G)$, with $\bsfA$ satisfying the ellipticity condition \eqref{A-ellipticity-Gamma} and $\bsfc$ being weakly $\divG$-free:
\begin{equation}\label{eq:c-div-free}
\int_\G \bsfc \cdot \nablaG w = 0, \qquad w \in \WW^{1,1}(\G).
\end{equation}
Then, for every $f \in (\WW^{1,p^*}_\#(\G))'$, there exists a unique $u \in \WW^{1,p}_\#(\G)$ such that 
\begin{equation}\label{eq:gsep-weakform-without-b-div-free-conv}
\int_\G \bsfA \nablaG u \cdot \nablaG v + (\bsfc \cdot \nablaG u) v = \langle f, v \rangle, \qquad \forall v \in \WW^{1,p^*}_\#(\G).
\end{equation}
Moreover, there exists a constant $C>0$, depending only on $\G$, $\bsfA$, $\bsfc$ and $p$, such that
\begin{equation}\label{eq:gsep-apriori-without-b-div-free-conv}
\|u\|_{1,p;\G} \leq C \|f\|_{(\WW^{1,p^*}_\#(\G))'}.
\end{equation}
This result is still valid for $\G$ of class $C^{0,1}$ and $\bsfA \in \bbL^\infty(\G)$ provided $p \in (2-\varepsilon, 2+\varepsilon)$ for $\varepsilon > 0$ sufficiently small depending only on $\G$ and $\bsfA$.

If in addition to the assumptions above, $\G$ is of class $C^{m+1,1}$ for some nonnegative integer $m$, and $(\bsfA, \bsfc) \in \bbC^{m,1}(\G) \times \bC_t^{m-1,1}(\G)$, then, for each $f \in \WW^{m,p}(\G)$, the solution $u \in \WW^{1,p}(\G)$ of \eqref{eq:gsep-weakform-without-b-div-free-conv} is in fact in $\WW^{m+2, p}(\G)$ and satisfies
\begin{equation}\label{gsep-a.e.-without-b-div-free-conv}
- \divG (\bsfA \nablaG u) + \bsfc \cdot \nablaG u = f, \qquad \text{a.e.~on $\G$}.
\end{equation}
Moreover, there is a constant $C >0$, depending only on $\G$, $\bsfA$, $\bsfc$ and $p$, such that
\begin{equation}\label{eq:gsep-higher-apriori-without-b-div-free-conv}
\|u\|_{m+2,p;\G} \leq C \|f\|_{m,p;\G}.
\end{equation}
\begin{proof}
\noindent\textbf{Well-posedness:}
We first tackle the case $p=2$.
Notice that (as in the proof of \autoref{prop:uniqueness-GSES}) for each $v \in \HH^1(\G)$ we have $\nablaG(uv) = v \nablaG u + u \nablaG v$ a.e.~and therefore from \eqref{eq:c-div-free} we deduce that
\begin{equation}\label{trilinear-skew}
\int_\G (\bsfc \cdot \nablaG u) v = - \int_\G (\bsfc \cdot \nablaG v) u, \qquad \forall u,v \in \HH^1(\G);
\end{equation}
hence $ \int_\G (\bsfc \cdot \nablaG u) u = 0$.
The well-posedness of \eqref{eq:gsep-weakform-without-b-div-free-conv} for $p=2$ and corresponding estimate \eqref{eq:gsep-apriori-without-b-div-free-conv}), are thus a consequence of Lax--Milgram theorem because the ellipticity condition \eqref{A-ellipticity-Gamma} and Poincaré--Friedrichs inequality in $\HH^1_\#(\G)$ \cite[Lemma 2]{BonitoDemlowNochetto2020} imply that \eqref{eq:gsep-weakform-without-b-div-free-conv} is coercive.

For $p \in (2,\infty)$, the unique solvability of problem \eqref{eq:gsep-weakform-without-b-div-free-conv} and estimate \eqref{eq:gsep-apriori-without-b-div-free-conv} follow from the case $p=2$ and a bootstrapping argument for the integrability of the solution based on \autoref{thm:div-form-wp} (invertibility of the operator $-\divG(\bsfA \nablaG \cdot)$ in $\WW^{1,p}_\#(\G)$). Therefore, the map $f \mapsto u$ is an isomorphism from $(\WW^{1,p^*}_\#(\G))'$ onto $\WW^{1,p}_\#(\G)$.

For $p \in (1,2)$ we argue by duality:
consider the operator $T : \WW^{1,p^*}_\#(\G) \rightarrow (\WW^{1,p}_\#(\G))'$ defined by $\langle Tv,u \rangle  := \int_\G \bsfA^T \nablaG v \cdot \nablaG u - (\bsfc \cdot \nablaG v)u$, which is an isomorphism because of the preceding step for $p^* > 2$.
Consequently, by \autoref{thm:inj-surj-rel} (characterization of injectivity and surjectivity), the adjoint operator $T^\mathrm{t}: \WW^{1,p}_\#(\G) \rightarrow (\WW^{1,p^*}_\#(\G))'$ defined by $\langle T^\mathrm{t} u,v \rangle := \langle Tv,u \rangle$ is also an isomorphism.
Moreover, because of identity \eqref{trilinear-skew} (which can be extended by a density argument to $u \in \WW^{1,p^*}_\#(\G)$ and $v \in \WW^{1,p^*}_\#(\G)$), we can write
\begin{equation*}
\langle T^\mathrm{t} u,v \rangle
= \langle Tv,u \rangle
= \int_\G \bsfA \nablaG u \cdot \nablaG v + (\bsfc \cdot \nablaG u) v,
\end{equation*}
which is the left-hand side of \eqref{eq:gsep-weakform-without-b-div-free-conv}.
This finishes the proof of the well-posedness of \eqref{eq:gsep-weakform-without-b-div-free-conv} for $p \in (1,\infty)$.

\medskip\noindent\textbf{Higher regularity:}
Finally, a bootstrapping argument based upon \autoref{thm:div-form-higreg} ($\WW^{m+2,p}$-regularity for the operator $-\divG(\bsfA \nablaG \cdot)$ on $\G$) akin to the proof of \autoref{thm:gsep-reg} (higher regularity for general scalar elliptic equations) leads to identity \eqref{gsep-a.e.-without-b-div-free-conv} and the a-priori estimate \eqref{eq:gsep-higher-apriori-without-b-div-free-conv}.
We omit further details.
\end{proof}
\end{thm}

\subsection{The biharmonic problem on \texorpdfstring{$\G$}{Γ}}\label{subsec:biharmonic}

The biharmonic problem is instrumental in modeling of plates and Stokes flow (via the stream function formulation).
In addition to having a self-standing interest, we shall also utilize it in \cite[\S6]{BenavidesNochettoShakipov2025-b} to prove higher-regularity estimates for solutions of tangential vector-valued PDEs driven by various definitions of the vector Laplacian.
Consider the following weak formulation for the biharmonic problem $\DeltaG^2 u = f$ on $\Gamma$: given $f \in (\WW^{2,p^*}_\#(\G))'$, find $u \in \WW^{2,p}_\#(\G)$ such that
\begin{equation}\label{eq:biharmonic}
\int_\G \DeltaG u \, \DeltaG v = \langle f, v \rangle, \qquad \forall v \in \WW^{2,p^*}_\#(\G).
\end{equation}
The next two theorems concern the well-posedness and higher regularity of \eqref{eq:biharmonic}.
\begin{thm}[well-posedness of the biharmonic problem] \label{thm:biharmonic-wp}
Let $p \in (1,\infty)$ and let $p^* := \frac{p}{p-1}$ be the dual exponent of $p$.
Assume $\G$ is of class $C^{1,1}$.
Then for every $f \in (\WW^{2,p^*}_\#(\G))'$, there exists a unique $u \in \WW^{2,p}_\#(\G)$ that solves \eqref{eq:biharmonic}.
Moreover, there exists a constant $C>0$, depending only on $\G$ and $p$, such that
\begin{equation*}
\|u\|_{2,p;\G} \leq C \|f\|_{(\WW^{2,p^*}_\#(\G))'}.
\end{equation*}
\end{thm}
\begin{proof}
Let $v \in \WW^{2,p^*}_\#(\G)$ be such that $\int_\G \DeltaG u \, \DeltaG v = 0$ for each $u \in \WW^{2,p}_\#(\G)$.
First, we observe that due to \autoref{thm:div-form-higreg} ($\WW^{m+2,p}$-regularity of the operator $-\divG(\bsfA \nablaG \cdot)$ on $\G$), this is equivalent to $\int_\G w \, \DeltaG v = 0$ for all $w \in \LL^p_\#(\G)$.
Consequently, we infer that $v = 0$. 

On the other hand, again due to \autoref{thm:div-form-higreg}, we also have
\begin{equation*}
\sup_{\substack{v \in \WW^{2,p^*}_\#(\G)\\ v \neq 0}} \frac{\int_\G \DeltaG u \, \DeltaG v}{\|v\|_{2,p^*;\G}}
\gtrsim
\sup_{\substack{v \in \WW^{2,p^*}_\#(\G)\\ v \neq 0}} \frac{\int_\G \DeltaG u \, \DeltaG v}{\|\DeltaG v\|_{0,p^*;\G}}
= \sup_{\substack{w \in \LL^{p^*}_\#(\G)\\ w \neq 0}} \frac{\int_\G w \, \DeltaG u}{\|w\|_{0,p^*;\G}} = \|\DeltaG u\|_{0,p;\G}
\gtrsim \|u\|_{2,p;\G}.
\end{equation*}
This verifies the conditions of \autoref{thm:BNB} (Banach--Ne\v{c}as--Babu\v{s}ka theorem), thus yielding the asserted well-posedness and a priori estimate.
\end{proof}
We would like to emphasize that $C^2$-regularity of $\G$ is not required for the aforementioned result to hold because it hinges on the $\WW^{2,p}(\G)$-regularity of the Laplace--Beltrami problem that only requires $\G$ to be $C^{1,1}$. For comparison, $\WW^{1,p}(\G)$-regularity for the Laplace--Beltrami problem requires $\G$ to be $C^1$, and would not be true for $\G$ of class $C^{0,1}$ unless $p$ is sufficiently close to $2$.
\begin{thm}[higher regularity for the biharmonic problem] \label{thm:biharmonic-reg}
Let $\G$ be of class $C^{m+2,1}$ for a nonnegative integer $m$ and let $f \in X_{m-1,p}$, where
\begin{equation*}
X_{m-1,p} := 
\begin{cases}
(\WW^{1,p^*}_\#(\G))', & m = 0,\\
\WW^{m-1,p}_\#(\G), & m \geq 1.
\end{cases}
\end{equation*}
Then, the solution $u \in \WW^{2,p}_\#(\G)$ of \eqref{eq:biharmonic} is in fact in $\WW^{m+3}_\#(\G)$ and there exists $C = C(\G, p) > 0$ such that
\begin{equation*}
\|u\|_{m+3,p;\G} \leq C \|f\|_{X_{m-1,p}}.
\end{equation*}
\end{thm}
\begin{proof}
By virtue of either \autoref{thm:div-form-wp} (invertibility of the operator $-\divG(\bsfA \nablaG \cdot)$ in $\WW^{1,p}_\#(\G)$) if $m=0$ or
\autoref{thm:div-form-higreg} ($\WW^{m+2,p}$-regularity of the operator $-\divG(\bsfA \nablaG \cdot)$ on $\G$) if $m \geq 1$,
there exists a unique $u_f \in \WW^{m+1,p}_\#(\G)$ that solves $\DeltaG u_f = f$ in the sense of \eqref{eq:div-form-W1p-weakform} and $\|u_f\|_{m+1,p;\G} \lesssim \|f\|_{X_{m-1,p}}$.
Consequently, by integrating by parts, problem \eqref{eq:biharmonic} can be equivalently stated as
\begin{equation*}
\int_\G \DeltaG u \; \DeltaG v = -\int_\G \nablaG u_f \cdot \nablaG v = \int_\G u_f \; \DeltaG v, \qquad \forall v \in \WW^{2,p^*}_\#(\G).
\end{equation*}
Again, by \autoref{thm:div-form-higreg}, we infer that $\DeltaG u = u_f$ a.e.~on $\G$ with $u_f \in \WW^{m+1,p}_\#(\G)$.
Therefore, because $\G$ is of class $C^{m+2,1}$, once again by \autoref{thm:div-form-higreg}, we deduce that $u \in \WW^{m+3,p}_\#(\G)$ and
\begin{equation*}
\|u\|_{m+3,p;\G} \lesssim \|u_f\|_{m+1,p;\G} \lesssim \|f\|_{X_{m-1,p}}.
\end{equation*}
This concludes the proof.
\end{proof}
Observe that we only invoked the $\WW^{m,p}(\G)$-regularity as well as the well-posedness in $\WW^{1,p}_\#(\G)$ of the Laplace--Beltrami problem to derive the well-posedness and higher regularity of the biharmonic problem \eqref{eq:biharmonic}.

\section*{Acknowledgements}
The authors are grateful to professor Giuseppe Savaré (Bocconi University), whose incisive observations on a previous version of this manuscript motivated us to streamline the overall arguments in the paper and improve its results.\looseness=-1


\appendix

\section{Variational problems in reflexive Banach spaces}\label{app:var-prob}
In this section we recall results on the well-posedness of variational problems in reflexive Banach spaces.
Classical references on the subject include \cite{ErnGuermond2004,Ciarlet2013,Brezis2011}.
We also mention the recent work \cite{Gatica2024} (in Spanish).

For any continuous linear operator $T: X \rightarrow Y'$ (that is, $T \in \sL(X,Y')$), where $X$ and $Y$ are Banach spaces, we let $T': Y'' \rightarrow X'$ denote its Banach adjoint defined by $T'(\sG) = \sG \circ T$ for each $\sG \in Y''$.
If $Y$ is reflexive, we also define $T^\mathrm{t}: Y \rightarrow X'$ by $T' \circ \cJ_Y$, where $\cJ_Y : Y \rightarrow Y''$ is the canonical embedding of $Y$ onto its bidual $Y''$.
In this case, $T$ and $T^\mathrm{t}$ are related via the identity
\begin{equation}\label{Banach-duality}
\langle T(x), y \rangle_{Y' \times Y} = \langle T^\mathrm{t}(y), x \rangle_{X' \times X}, \qquad \forall x \in X, \, y \in Y.
\end{equation}
The relation between the injectivity and surjectivity of the operators $T$ and $T^\mathrm{t}$ is given by the following result \cite[Section 2.7]{Brezis2011}.
\begin{thm}[characterization of injectivity and surjectivity]\label{thm:inj-surj-rel}
Let $X$ and $Y$ be two Banach spaces with $Y$ reflexive, and consider $T \in \sL(X,Y')$.
Then, the following set of equivalences hold true
\begin{center}
\begin{tabular}{lccccc}
i) & $T$ is surjective & $\Longleftrightarrow$ & \begin{minipage}{3cm} $T^\mathrm{t}$ is injective and has closed range \end{minipage} & $\Longleftrightarrow$ &
\begin{minipage}[c]{5cm} there is $\alpha>0$ such that $$\sup_{\substack{x \in X\\ x \neq 0}} \frac{\langle T(x), y \rangle}{\|x\|} \geq \alpha \|y\|, \quad \forall y \in Y$$
\end{minipage}\\
& & & & & \\
ii) & \begin{minipage}{3cm} $T$ is injective and has closed range \end{minipage} & $\Longleftrightarrow$ & $T^\mathrm{t}$ is surjective & $\Longleftrightarrow$ & \begin{minipage}[c]{3.8cm} there is $\wt\alpha>0$ such that $$\sup_{\substack{y \in Y\\ y \neq 0}} \frac{\langle T(x), y \rangle}{\|y\|} \geq \wt\alpha \|x\|, \quad \forall x \in X$$.
\end{minipage}
\end{tabular}
\end{center}
\end{thm}
A direct consequence of \autoref{thm:inj-surj-rel} is the celebrated Banach--Ne\v{c}as--Babu\v{s}ka theorem.

\begin{thm}[Banach--Ne\v{c}as--Babu\v{s}ka]\label{thm:BNB}
Let $X$ and $Y$ be two Banach spaces with $Y$ reflexive, and consider $T \in \sL(X,Y')$.
Then, $T$ is invertible (and hence $T^{-1} \in \sL(Y',X)$ by the Bounded Inverse Theorem) if and only if one of the following pair of conditions hold

\vspace*{0.5cm}
\begin{minipage}{0.47\textwidth}
\begin{enumerate}
\item[i)]\label{it:cond-i} If $x \in X$ is such that $\langle T(x), y \rangle = 0$ for each $y \in Y$, then $x = 0$.
\item[ii)]\label{it:cond-ii} There is $\alpha>0$ such that $$\sup_{\substack{x \in X\\ x \neq 0}} \frac{\langle T(x), y \rangle}{\|x\|} \geq \alpha \|y\|, \quad \forall y \in Y.$$
\end{enumerate}
\end{minipage}%
\begin{minipage}{0.47\textwidth}
\begin{enumerate}
\item[i)']\label{it:cond-i'} If $y \in Y$ is such that $\langle T(x), y \rangle = 0$ for each $x \in X$, then $y = 0$.
\item[ii)']\label{it:cond-ii'} There is $\wt\alpha>0$ such that $$\sup_{\substack{y \in Y\\ y \neq 0}} \frac{\langle T(x), y \rangle}{\|y\|} \geq \wt\alpha \|x\|, \quad \forall x \in X.$$
\end{enumerate}
\end{minipage}

\noindent Moreover, if $T$ is invertible, then the ``optimal'' constants $\alpha$ and $\wt\alpha$ coincide, that is
\begin{equation*}
\alpha_{\text{opt}} = \inf_{\substack{y \in Y\\ y \neq 0}}\sup_{\substack{x \in X\\ x \neq 0}} \frac{\langle T(x), y \rangle}{\|x\| \|y\|} = \inf_{\substack{x \in X\\ x \neq 0}}\sup_{\substack{y \in Y\\ y \neq 0}} \frac{\langle T(x), y \rangle}{\|x\| \|y\|},
\end{equation*}
and it also holds that
\begin{equation*}
\|T^{-1}\|_{\sL(Y',X)} = \|(T^\mathrm{t})^{-1}\|_{\sL(X',Y)} = \frac{1}{\alpha_{\text{opt}}}.
\end{equation*}
\end{thm}
The development of the functional analytical framework presented in \autoref{thm:inj-surj-rel} and \autoref{thm:BNB} is the consequence of decades of research built upon the seminal work by Ne\v{c}as \cite{Necas1961,Necas1962}, and later enriched by Hayden \cite{Hayde1968}, Babu\v{s}ka \cite{Babuska1970}, Babu\v{s}ka and Aziz \cite{BabuskaAziz1972} and Simader \cite{Simader1972};
we refer to \cite{Saito2018} for a thorough discussion on its historical development.
An exhaustive and self-contained derivation of \autoref{thm:BNB} can be found in the classical reference \cite[Chapter 2]{ErnGuermond2004} (see also \cite[Capítulo 5]{Gatica2024}).

\begin{remark}[continuous inverse of a surjective operator]\label{rem:inv-surj}
If $T \in \sL(X,Y')$ is surjective (but not necessarily injective), then there exists a constant $\beta > 0$ such that for $F \in Y'$ there exists $x \in X$ such that $T(x) = F$ and $\|x\| \leq \frac{1}{\beta} \|F\|$.
If $X$ is reflexive, then $\beta$ can be taken as equal to the constant $\alpha$ given by item i) of \autoref{thm:inj-surj-rel}.
\end{remark}

\section{Regularity estimates in flat domains}\label{app:regularity-flat}
This section contains results on higher Sobolev regularity of elliptic PDEs in flat domains, which are essential for our arguments in \autoref{sec:scal-ell}.

The following is a particularization of \cite[Lemma 2]{DolzmannMuller1995}.
\begin{proposition}[well-posedness of general scalar elliptic equations]\label{prop:gen-el-sys}
Let $\Omega$ be an open connected subset of $\RR^d$ with a $C^1$ boundary and let $\bA \in \bbC(\overline{\Omega})$ be uniformly elliptic on $\Omega$, i.e.,~there exists $\Lambda > 0$ such that
\begin{equation*}
\bxi \cdot \bA(\bx) \bxi \geq \Lambda |\bxi|^2, \qquad \forall \xi \in \RR^d, \quad \forall \bx \in \Omega.
\end{equation*}
Moreover, let $p \in (1,\infty)$ and $q \in (1,d)$.
Then, for each $\bf \in \bL^p(\Omega)$ and $f \in \LL^q(\Omega)$, the problem
\begin{equation}\label{eq:gen-el-sys}
\begin{aligned}
-\div(\bA \nabla u) & = \div \bf + f, \qquad \text{in $\Omega$},\\
u & = 0, \qquad \text{on $\partial\Omega$}
\end{aligned}
\end{equation}
has a weak solution $u \in \WW_0^{1,s}(\Omega)$, where $s := \min\{p, \frac{dq}{d-q}\}$.
Moreover, there exists a constant $C>0$ depending only on $\Omega$, $\bA$, $p$ and $q$, such that the following a-priori estimate holds
\begin{equation*}
\|u\|_{1,s;\Omega} \leq C \left( \|\bf\|_{0,p;\Omega} + \|f\|_{0,q;\Omega}  \right).
\end{equation*}
Moreover, $u$ is unique in the sense that if for some $t \in (1,\infty)$ there is a weak solution $\wt u \in \WW_0^{1,t}(\Omega)$ of \eqref{eq:gen-el-sys}, then $u = \wt u$ a.e.
\end{proposition}

\begin{remark}[Meyers]\label{app:rem1}
From \cite[Remark 2]{DolzmannMuller1995} and \cite[Theorem 1]{Meyers1963} we know that if $\Omega$ is only of class $C^{0,1}$ and $\bA$ only lies $\bL^\infty(\Omega)$, then there exists $\varepsilon>0$, depending only on $\Omega$ and $\bA$ such that, \autoref{prop:gen-el-sys} also holds as long as $p,s \in (2-\varepsilon,2+\varepsilon)$.
\end{remark}

The following result concerning the higher regularity of solutions to \eqref{eq:gen-el-sys} is an easy consequence of the $\LL^p$-based theory found in \cite[Chapter 9]{GilbargTrudinger2001}.
For completeness, we provide a short proof below.
\begin{proposition}[higher regularity]\label{app:prop-hi-reg}
In addition to the assumptions of \autoref{prop:gen-el-sys}, assume there exists $m \in \NN$ such that $\Omega$ is of class $C^{m,1}$, $\bA \in \bbC^{m-1,1}(\overline\Omega)$, $\bf \in \bW^{m,p}(\Omega)$ and $f \in \WW^{m,q}(\Omega)$.
Then, $u \in \WW^{m+1,s}(\Omega)$ and
\begin{equation}\label{eq:app-hig-est}
\|u\|_{m+1,s;\Omega} \leq \wt C \left( \|\bf\|_{m,p;\Omega} + \|f\|_{m,q;\Omega}  \right),
\end{equation}
where $s: = \min\{p, \frac{dq}{d-q}\}$ and $\wt C >0$ depends only on $\Omega$, $\bA$, $p$ and $q$.
\begin{proof}
We first notice that for $u \in \WW^{2,s}(\Omega)$ we can write $\div(\bA \nabla u) = \bA : \nabla^2 u + \div(\bA^T) \cdot \nabla u = : Lu$, with $\div(\bA^T)$ lying in $\bL^\infty(\Omega)$ if $m=1$ or $\bC^{m-2,1}(\overline{\Omega})$ if $m \geq 2$.
Consequently, it follows from \cite[Theorem 9.15]{GilbargTrudinger2001} for $m=1$ and \cite[Theorem 9.19]{GilbargTrudinger2001} for $m \geq 2$ that $L$ is a linear and continuous bijection from $\WW^{m+1,s}(\Omega) \cap \WW^{1,s}_0(\Omega)$ onto $\WW^{m-1,s}(\Omega)$.
By the Bounded Inverse Theorem we deduce that $\|L^{-1} g\|_{m+1,s;\Omega} \lesssim \|g\|_{m-1,s;\Omega}$ for each $g \in \WW^{m-1,s}(\Omega)$.
In particular, choosing $g := -(\div \bf + f)$, which satisfies $\|g\|_{m-1,s;\Omega} \lesssim \|\bf\|_{m,s;\Omega} + \|f\|_{m-1,s;\Omega} \lesssim \|\bf\|_{m,p;\Omega} + \|f\|_{m,q;\Omega}$, implies that the solution $u$ of \eqref{eq:gen-el-sys} is in fact in $\WW^{m+1,s}(\Omega)$ and satisfies the a-priori estimate \eqref{eq:app-hig-est}.
\end{proof}
\end{proposition}

\bibliographystyle{abbrv}
\bibliography{references}
\end{document}